\documentclass[12pt]{amsart}
\usepackage{amssymb}
\usepackage[all]{xy}
\usepackage{tikz}
\usepackage{amsmath}
\numberwithin{equation}{section}
\input xy
\xyoption{all}

\newtheorem{lemma}{Lemma}[section]
\newtheorem{corollary}[lemma]{Corollary}
\newtheorem{theorem}[lemma]{Theorem}
\newtheorem{proposition}[lemma]{Proposition}

\theoremstyle{definition}
\newtheorem{remark}[lemma]{Remark}
\newtheorem{definition}[lemma]{Definition}

\DeclareMathOperator{\Mod}{Mod}
\DeclareMathOperator{\modd}{mod}

\DeclareMathOperator{\Hom}{Hom}

\DeclareMathOperator{\Ext}{Ext}
\DeclareMathOperator{\nExt}{nExt}

\DeclareMathOperator{\Ker}{Ker}
\DeclareMathOperator{\Imm}{Im}
\DeclareMathOperator{\Eff}{Eff}
\DeclareMathOperator{\eff}{eff}

\newtheorem*{theorem a*}{Theorem A}
\newtheorem*{theorem b*}{Theorem B}

\newcounter{diagram}
\numberwithin{diagram}{section}

\newenvironment{diagram}
  {\stepcounter{diagram}\par\smallskip\noindent\begin{minipage}{\linewidth}\centering}
  {\par Diagram~\thediagram\end{minipage}\par\smallskip}

\begin{document}

\title{$n\mathbb{Z}$-abelian and $n\mathbb{Z}$-exact categories}

\author{Ramin Ebrahimi}
\address{Department of Pure Mathematics\\
Faculty of Mathematics and Statistics\\
University of Isfahan\\
P.O. Box: 81746-73441, Isfahan, Iran\\ and School of Mathematics, Institute for Research in Fundamental Sciences (IPM), P.O. Box: 19395-5746, Tehran, Iran}
\email{ramin69@sci.ui.ac.ir / r.ebrahimi@ipm.ir}

\author{Alireza Nasr-Isfahani}
\address{Department of Pure Mathematics\\
Faculty of Mathematics and Statistics\\
University of Isfahan\\
P.O. Box: 81746-73441, Isfahan, Iran\\ and School of Mathematics, Institute for Research in Fundamental Sciences (IPM), P.O. Box: 19395-5746, Tehran, Iran}
\email{nasr$_{-}$a@sci.ui.ac.ir / nasr@ipm.ir}



\begin{abstract}
In this paper we introduce $n\mathbb{Z}$-abelian and $n\mathbb{Z}$-exact categories by axiomatising properties of $n\mathbb{Z}$-cluster tilting subcategories. We study this categories and show that every $n\mathbb{Z}$-cluster tilting subcategory of an abelian (resp., exact) category has a natural structure of an $n\mathbb{Z}$-abelian (resp., $n\mathbb{Z}$-exact) category. Also we show that every small $n\mathbb{Z}$-abelian category arise in this way, and discuss the problem for $n\mathbb{Z}$-exact categories.
\end{abstract}

\maketitle

%%%%%%%%%%%%%%%%%%%%%%%%%%%%%%%%%%%%%%

\section{Introduction}
Higher Auslander-Reiten theory was introduced by Iyama in \cite{I2,I1}. It deals with $n$-cluster tilting subcategories of abelian and exact categories. Jasso in \cite{J} introduced
$n$-abelian and $n$-exact categories as a higher-dimensional analogue of abelian and exact
categories, that are axiomatisation of $n$-cluster tilting subcategories. Jasso proved that
each $n$-cluster tilting subcategory of an abelian (resp., exact) category is $n$-abelian (resp., $n$-exact). Also every $n$-abelian category has been shown to arise in this way \cite{EN, Kv}.

Special kind of $n$-cluster tilting subcategories, called $n\mathbb{Z}$-cluster tilting subcategories have nicer properties than general $n$-cluster tilting subcategories. An $n$-cluster tilting subcategory is said to be an {\em $n\mathbb{Z}$-cluster tilting subcategory} if satisfies the following additional condition.
\begin{itemize}
\item If $\Ext^k(\mathcal{M},\mathcal{M})\neq 0$ then $k\in n\mathbb{Z}$.
\end{itemize}

In this paper we give the following characterisation of $n\mathbb{Z}$-cluster tilting subcategories of exact categories. We refer the reader to Section 2 for the definitions of $n$-cluster tilting subcategory and $n$-exact sequence.

\begin{theorem}\label{1.1}
Let $\mathcal{M}$ be an $n$-cluster tilting subcategory of an exact category $\mathcal{E}$. The following conditions are equivalent.
\begin{itemize}
\item[(1)]
$\mathcal{M}$ is an $n\mathbb{Z}$-cluster tilting subcategory.
\item[(2)]
For every $X\in \mathcal{M}$ and every $n$-exact sequence $Y:Y^0\rightarrow Y^1\rightarrow \cdots Y^n\rightarrow Y^{n+1}$ the following induced sequence of abelian groups is exact.
\begin{align*}
0&\rightarrow \Hom_{\mathcal{E}}(X,Y^0)\rightarrow \Hom_{\mathcal{E}}(X,Y^1)\rightarrow\cdots\rightarrow \Hom_{\mathcal{E}}(X,Y^n)\rightarrow\Hom_{\mathcal{E}}(X,Y^{n+1}) \\
&\rightarrow \Ext_{\mathcal{E}}^n(X,Y^0)\rightarrow \Ext_{\mathcal{E}}^n(X,Y^1)\rightarrow\cdots\rightarrow \Ext_{\mathcal{E}}^n(X,Y^n)\rightarrow\Ext_{\mathcal{E}}^n(X,Y^{n+1}) \\
&\rightarrow \Ext_{\mathcal{E}}^{2n}(X,Y^0)\rightarrow \Ext_{\mathcal{E}}^{2n}(X,Y^1)\rightarrow\cdots\rightarrow \Ext_{\mathcal{E}}^{2n}(X,Y^n)\rightarrow\Ext_{\mathcal{E}}^{2n}(X,Y^{n+1}) \\
&\rightarrow \cdots.
\end{align*}
\item[(3)]
For every $Y\in \mathcal{M}$ and every $n$-exact sequence $X:X^0\rightarrow X^1\rightarrow \cdots X^n\rightarrow X^{n+1}$ the following induced sequence of abelian groups is exact.
\begin{align*}
0&\rightarrow \Hom_{\mathcal{E}}(X^{n+1},Y)\rightarrow \Hom_{\mathcal{E}}(X^n,Y)\rightarrow\cdots\rightarrow \Hom_{\mathcal{E}}(X^1,Y)\rightarrow\Hom_{\mathcal{E}}(X^0,Y) \\
&\rightarrow \Ext_{\mathcal{E}}^n(X^{n+1},Y)\rightarrow \Ext_{\mathcal{E}}^n(X^n,Y)\rightarrow\cdots\rightarrow \Ext_{\mathcal{E}}^n(X^1,Y)\rightarrow\Ext_{\mathcal{E}}^n(X^0,Y) \\
&\rightarrow \Ext_{\mathcal{E}}^{2n}(X^{n+1},Y)\rightarrow \Ext_{\mathcal{E}}^{2n}(X^n,Y)\rightarrow\cdots\rightarrow \Ext_{\mathcal{E}}^{2n}(X^1,Y)\rightarrow\Ext_{\mathcal{E}}^{2n}(X^0,Y) \\
&\rightarrow \cdots.
\end{align*}
\end{itemize}
\end{theorem}

Iyama and Jasso in \cite[Definition-Proposition 2.15]{IJ} proved the above theorem for an $n$-cluster tilting
subcategory $\mathcal{M}\subseteq \modd(\mathcal{A})$, where $\mathcal{A}$ is a dualizing $R$-variety. They also showed that $\mathcal{M}\subseteq \modd(\mathcal{A})$ is an $n\mathbb{Z}$-cluster tilting subcategory if and only if $\mathcal{M}$ closed under $n$th syzygies if and only if $\mathcal{M}$ closed under $n$th cosyzygies. The proof of Iyama and Jasso of this theorem in \cite{IJ}, heavily biased on the enough projectives and enough injectives properties of $\mathcal{A}$. But in Theorem \ref{1.1}, we don't have such assumptions.

Iyama in \cite[Appendix A]{I1} showed that for an $n$-cluster tilting subcategory $\mathcal{M}\subseteq \modd\Lambda$, where $\Lambda$ is an Artin algebra and for every two objects $M,N\in \mathcal{M}$, every element in $\Ext^n_{\Lambda}(M,N)$ is Yoneda equivalent to a unique (up to homotopy) $n$-fold extension of $M$ by $N$ with terms in $\mathcal{M}$. In the following theorem we give the following more general version of this result for any $n$-cluster tilting subcategory of an exact category $\mathcal{E}$.

\begin{theorem}\label{1.2}
Let $\mathcal{M}$ be an $n$-cluster tilting subcategory of an exact category $\mathcal{E}$ and
\begin{equation}
\xi:0\rightarrow X^0\rightarrow E^1\rightarrow E^2\rightarrow\cdots\rightarrow E^n\rightarrow X^{n+1}\rightarrow 0,\notag
\end{equation}
with $X^0,X^{n+1}\in \mathcal{M}$ be an acyclic sequence in $\mathcal{E}$. Then there is a unique (up to homotopy) $n$-exact sequence
\begin{equation}
0\rightarrow X^0\rightarrow X^1\rightarrow X^2\rightarrow\cdots\rightarrow X^n\rightarrow X^{n+1}\rightarrow 0,\notag
\end{equation}
Yoneda equivalent to $\xi$.

If moreover $\mathcal{M}$ be an $n\mathbb{Z}$-cluster tilting subcategory and
\begin{equation}
\xi:0\rightarrow X^0\overset{f^0}{\longrightarrow} E^1\overset{f^1}{\longrightarrow} \cdots\overset{f^{kn-2}}{\longrightarrow} E^{kn-1}\overset{f^{kn-1}}{\longrightarrow} E^{kn}\overset{f^{kn}}{\longrightarrow} X^{kn+1}\rightarrow 0,\notag
\end{equation}
be a $kn$-fold extension with $X^0,X^{kn+1}\in \mathcal{M}$, then $\xi$ is Yoneda equivalent to splicing of $k$, $n$-exact sequences.
\end{theorem}

Let $\mathcal{M}$ be an $n$-exact category. For every two object $X,Y\in \mathcal{M}$ we can define $\nExt^1(X,Y)$ as Yoneda equivalence classes of $n$-exact sequences starting from $Y$ and ending with $X$. Also for a positive integer $k$, $\nExt^k(X,Y)$ is defined as Yoneda equivalence classes of $k$-fold $n$-extensions. We refer the reader to Section 2 for the definition of $n$-exact category and to Section 5 for the definitions of $\nExt^1(X,Y)$ and $\nExt^k(X,Y)$.
Motivated by Theorems \ref{1.1} and \ref{1.2} we define $n\mathbb{Z}$-exact (resp. $n\mathbb{Z}$-abelian) categories as axiomatisation of $n\mathbb{Z}$-cluster tilting subcategories of exact (resp. abelian) categories (See Definition \ref{5.2}). Let $\mathcal{M}$ be an $n$-exact (resp. $n$-abelian) category. We say that $\mathcal{M}$ is an {\em$n\mathbb{Z}$-exact} (resp. {\em$n\mathbb{Z}$-abelian}) category if for every object $X\in\mathcal{M}$, $\nExt^*(X,-)$ and $\nExt^*(-,X)$ induce long exact sequences for all $n$-exact sequences.

We show that an $n$-cluster tilting subcategory of an abelian (resp. exact) category is an $n\mathbb{Z}$-exact (resp. $n\mathbb{Z}$-abelian) category if and only if it is an $n\mathbb{Z}$-cluster tilting subcategory of an abelian (resp. exact) category. This shows that $n\mathbb{Z}$-exact (resp. $n\mathbb{Z}$-abelian) categories are good axiomatisation of $n\mathbb{Z}$-cluster tilting subcategories of exact (resp. abelian) categories.

The paper is organized as follows.
In section 2 we recall the definitions of $n$-cluster tilting subcategories of abelian and exact categories, $n$-exact and $n$-abelian categories and some of their basic properties. In section 3 we recall the Gabriel-Quillen embedding for $n$-exact categories and prove that it has expected properties as the classical Gabriel-Quillen embedding.
In section 4 we study $n\mathbb{Z}$-cluster tilting subcategories and prove Theorem \ref{1.1} and Theorem \ref{1.2}.
In section 5, motivated by the results of previous sections we introduce $n\mathbb{Z}$-abelian and $n\mathbb{Z}$-exact categories, and we prove that every small $n\mathbb{Z}$-abelian category is equivalent to an $n\mathbb{Z}$-cluster tilting subcategories of an abelian category. For small $n\mathbb{Z}$-exact categories we prove a similar result using the Gabriel-Quillen embedding.

\subsection{Notation}
Throughout this paper, unless otherwise stated, $n$ always denotes a fixed positive integer. All categories we consider are
assumed to be additive and by subcategory we mean full subcategory
which is closed under isomorphisms.

%%%%%%%%%%%%%%%%%%%%%%%%%%%%%%%%%%%%%%

\section{preliminaries}
In this section we recall the definitions of $n$-exact category, $n$-abelian category and $n$-cluster tilting subcategory. Also we recall some basic results that we need in the rest of the paper. For further information the readers are referred to \cite{I2, I3, I1, J}.

\subsection{$n$-exact categories}
Let $\mathcal{M}$ be an additive category and $f:A\rightarrow B$ a morphism in $\mathcal{M}$. A {\em weak cokernel} of $f$ is a morphism $g:B\rightarrow C$ such that for all $C^{\prime} \in \mathcal{M}$  the sequence of abelian groups
\begin{equation}
\Hom(C,C')\overset{(g,C')}{\longrightarrow} \Hom(B,C')\overset{(f,C')}{\longrightarrow} \Hom(A,C') \notag
\end{equation}
is exact. The concept of {\em weak kernel} is defined dually.

Let $d^0:X^0 \rightarrow X^1$ be a morphism in $\mathcal{M}$. An {\em$n$-cokernel} of $d^0$ is a sequence
\begin{equation}
(d^1, \ldots, d^n): X^1 \overset{d^1}{\rightarrow} X^2 \overset{d^2}{\rightarrow}\cdots \overset{d^{n-1}}{\rightarrow} X^n \overset{d^n}{\rightarrow} X^{n+1} \notag
\end{equation}
of objects and morphisms in $\mathcal{M}$ such that for each $Y\in \mathcal{M}$
the induced sequence of abelian groups
\begin{align}
0 \rightarrow \Hom(X^{n+1},Y) \rightarrow \Hom(X^n,Y) \rightarrow\cdots\rightarrow \Hom(X^1,Y) \rightarrow \Hom(X^0,Y) \notag
\end{align}
is exact. Equivalently, the sequence $(d^1, \ldots, d^n)$ is an $n$-cokernel of $d^0$ if for all $1\leq k\leq n-1$
the morphism $d^k$ is a weak cokernel of $d^{k-1}$, and $d^n$ is moreover a cokernel of $d^{n-1}$ \cite[Definition 2.2]{J}. The concept of {\em$n$-kernel} of a morphism is defined dually.
\begin{definition}\cite[Definition 2.4]{L}\label{d1}
Let $\mathcal{M}$ be an additive category. A {\em left $n$-exact sequence} in $\mathcal{M}$ is a complex
\begin{equation}
X^0 \overset{d^0}{\rightarrow} X^1 \overset{d^1}{\rightarrow} \cdots \overset{d^{n-1}}{\rightarrow} X^n \overset{d^n}{\rightarrow} X^{n+1} \notag
\end{equation}
such that $(d^0, \ldots, d^{n-1})$ is an $n$-kernel of $d^n$. The concept of {\em right $n$-exact sequence} is defined dually. An {\em$n$-exact sequence} is a sequence which is both a right $n$-exact sequence and a left $n$-exact sequence.
\end{definition}

\begin{definition}$($\cite[Definition 3.1]{J}$)$
An {\em $n$-abelian} category is an additive category $\mathcal{M}$ which satisfies the following axioms.
\begin{itemize}
\item[(i)]
The category $\mathcal{M}$ is idempotent complete.
\item[(ii)]
Every morphism in $\mathcal{M}$ has an $n$-kernel and an $n$-cokernel.
\item[(iii)]
For every monomorphism $d^0:X^0 \rightarrow X^1$ in $\mathcal{M}$ and for every $n$-cokernel $(d^1, \ldots, d^n)$ of $d^0$, the following sequence is $n$-exact:
\begin{equation}
X^0 \overset{d^0}{\rightarrow} X^1 \overset{d^1}{\rightarrow} \cdots \overset{d^{n-1}}{\rightarrow} X^n \overset{d^n}{\rightarrow} X^{n+1}. \notag
\end{equation}
\item[(iv)]
For every epimorphism $d^n:X^n \rightarrow X^{n+1}$ in $\mathcal{M}$ and for every $n$-kernel $(d^0, \ldots, d^{n-1})$ of $d^n$, the following sequence is $n$-exact:
\begin{equation}
X^0 \overset{d^0}{\rightarrow} X^1 \overset{d^1}{\rightarrow} \cdots \overset{d^{n-1}}{\rightarrow} X^n \overset{d^n}{\rightarrow} X^{n+1}. \notag
\end{equation}
\end{itemize}
\end{definition}

Let $X$ and $Y$ be two $n$-exact sequences. We remained that a morphism $f:X\rightarrow Y$ of $n$-exact sequences is a morphism of complexes.
A morphism $f:X\rightarrow Y$ of $n$-exact sequences is called a{\em weak isomorphism} if $f^k$
and $f^{k+1}$ are isomorphisms for some $k\in \{0, 1, \cdots, n + 1\}$ with $n + 2 := 0$ \cite[Definition 4.1]{J}.

Let
\begin{center}
\begin{tikzpicture}
\node (X1) at (-4,1) {$X$};
\node (X2) at (-2,1) {$X^0$};
\node (X3) at (0,1) {$X^1$};
\node (X4) at (2,1) {$\ldots$};
\node (X5) at (4,1) {$X^{n-1}$};
\node (X6) at (6,1) {$X^n$};
\node (X7) at (-4,-1) {$Y$};
\node (X8) at (-2,-1) {$Y^0$};
\node (X9) at (0,-1) {$Y^1$};
\node (X10) at (2,-1) {$\ldots$};
\node (X11) at (4,-1) {$Y^{n-1}$};
\node (X12) at (6,-1) {$Y^n$};
\draw [->,thick] (X1) -- (X7) node [midway,left] {$f$};
\draw [->,thick] (X2) -- (X8) node [midway,left] {$f^0$};
\draw [->,thick] (X3) -- (X9) node [midway,left] {$f^1$};
\draw [->,thick] (X5) -- (X11) node [midway,left] {$f^{n-1}$};
\draw [->,thick] (X6) -- (X12) node [midway,left] {$f^n$};
\draw [->,thick] (X2) -- (X3) node [midway,above] {$d_X^0$};
\draw [->,thick] (X3) -- (X4) node [midway,above] {$d_X^1$};
\draw [->,thick] (X4) -- (X5) node [midway,above] {$d_X^{n-2}$};
\draw [->,thick] (X5) -- (X6) node [midway,above] {$d_X^{n-1}$};
\draw [->,thick] (X8) -- (X9) node [midway,below] {$d_Y^0$};
\draw [->,thick] (X9) -- (X10) node [midway,below] {$d_Y^1$};
\draw [->,thick] (X10) -- (X11) node [midway,below] {$d_Y^{n-2}$};
\draw [->,thick] (X11) -- (X12) node [midway,below] {$d_Y^{n-1}$};
\end{tikzpicture}
\end{center}
be a morphism of complexes in an additive category. Recall that the {\em mapping cone}
$C = C( f )$ of $f$
is the complex
\begin{equation} \label{Cone}
X^0 \xrightarrow{d_C^{-1}} X^1\oplus Y^0 \xrightarrow{d_C^0} \cdots \xrightarrow{d_C^{n-2}} X^n\oplus Y^{n-1}\xrightarrow{d_C^{n-1}} Y^n,
\end{equation}
where
\begin{center}
$d_C^k:=\begin{pmatrix}
-d_X^{k+1}& 0\\\\
f^{k+1} & d_Y^k
\end{pmatrix}: X^{k+1}\oplus Y^k\to X^{k+2} \oplus Y^{k+1}$
\end{center}
for each
$k\in \{-1, 0, \ldots, n-1\}$.
In particular
$d_C^{-1}=\begin{pmatrix}
-d_X^0\\
f^0
\end{pmatrix}$
and
$d_C^{n-1}=\begin{pmatrix}
f^{n} & d_Y^{n-1}
\end{pmatrix}$.
\begin{itemize}
\item
The above diagram is called an {\em$n$-pullback} of
$Y$
along
$f^n$
if the complex \eqref{Cone} is a left $n$-exact sequence.
\item
The above diagram is an {\em$n$-pushout} of
$X$
along
$f^0$
if the complex \eqref{Cone} is a right $n$-exact sequence \cite[Definition 2.11]{J}.
\end{itemize}
\begin{definition}$($\cite[Definition 4.2]{J}$)$
Let $\mathcal{M}$ be an additive category. An {\em$n$-exact structure}
on $\mathcal{M}$ is a class $\mathcal{X}$ of $n$-exact sequences in $\mathcal{M}$, closed under weak isomorphisms of $n$-exact
sequences, and which satisfies the following axioms:
\begin{itemize}
\item[$(E0)$]
The sequence $0\rightarrowtail 0\rightarrow \cdots\rightarrow 0\twoheadrightarrow 0$ is an $\mathcal{X}$-admissible $n$-exact sequence.
\item[$(E1)$]
The class of $\mathcal{X}$-admissible monomorphisms is closed under composition.
\item[$(E1^{op})$]
The class of $\mathcal{X}$-admissible epimorphisms is closed under composition.
\item[$(E2)$]
For each $\mathcal{X}$-admissible $n$-exact sequence $X$ and each morphism $f:X^0\rightarrow Y^0$, there exists an $n$-pushout diagram of $(d_X^0,\cdots , d_X^{n-1})$ along $f$ such that $d_Y^0$ is an $\mathcal{X}$-admissible monomorphism. The situation is illustrated in the following commutative diagram:
\begin{center}
\begin{tikzpicture}
\node (X1) at (-4,1) {$X^0$};
\node (X2) at (-2,1) {$X^1$};
\node (X3) at (0,1) {$\cdots$};
\node (X4) at (2,1) {$X^n$};
\node (X5) at (4,1) {$X^{n+1}$};
\node (X6) at (-4,-1) {$Y^0$};
\node (X7) at (-2,-1) {$Y^1$};
\node (X8) at (0,-1) {$\cdots$};
\node (X9) at (2,-1) {$X^n$};
\draw [>->,thick] (X1) -- (X2) node [midway,above] {$d_X^0$};
\draw [->,thick] (X2) -- (X3) node [midway,above] {$d_X^1$};
\draw [->,thick] (X3) -- (X4) node [midway,above] {$d_X^{n-1}$};
\draw [->>,thick] (X4) -- (X5) node [midway,above] {$d_X^n$};
\draw [>->,thick,dashed] (X6) -- (X7) node [midway,above] {$d_Y^0$};
\draw [->,thick,dashed] (X7) -- (X8) node [midway,above] {$d_Y^1$};
\draw [->,thick,dashed] (X8) -- (X9) node [midway,above]{$d_Y^{n-1}$};
\draw [->,thick] (X1) -- (X6) node [midway,left] {$f$};
\draw [->,thick,dashed] (X2) -- (X7) node [midway,left] {};
\draw [->,thick,dashed] (X4) -- (X9) node [midway,left] {};
\end{tikzpicture}
\end{center}

\item[$(E2^{op})$]
For each $\mathcal{X}$-admissible $n$-exact sequence $Y$ and each morphism $g:X^{n+1}\rightarrow Y^{n+1}$, there exists an $n$-pullback diagram of $(d_Y^1,\cdots , d_Y^{n})$ along $g$ such that $d_X^{n}$ is an $\mathcal{X}$-admissible epimorphism. The situation is illustrated in the following commutative diagram:
\begin{center}
\begin{tikzpicture}
%\node (X1) at (-4,1) {$X^0$};
\node (X2) at (-2,1) {$X^1$};
\node (X3) at (0,1) {$\cdots$};
\node (X4) at (2,1) {$X^n$};
\node (X5) at (4,1) {$X^{n+1}$};
\node (X6) at (-4,-1) {$Y^0$};
\node (X7) at (-2,-1) {$Y^1$};
\node (X8) at (0,-1) {$\cdots$};
\node (X9) at (2,-1) {$Y^n$};
\node (X10) at (4,-1) {$Y^{n+1}$};
%\draw [->,thick] (X1) -- (X2) node [midway,above] {$d_X^0$};
\draw [->,thick,dashed] (X2) -- (X3) node [midway,above] {$d_X^1$};
\draw [->,thick,dashed] (X3) -- (X4) node [midway,above] {$d_X^{n-1}$};
\draw [->>,thick,dashed] (X4) -- (X5) node [midway,above] {$d_X^n$};
\draw [>->,thick] (X6) -- (X7) node [midway,above] {$d_Y^0$};
\draw [->,thick] (X7) -- (X8) node [midway,above] {$d_Y^1$};
\draw [->,thick] (X8) -- (X9) node [midway,above] {$d_Y^{n-1}$};
\draw [->>,thick] (X9) -- (X10) node [midway,above] {$d_Y^{n}$};
\draw [->,thick] (X5) -- (X10) node [midway,right] {$g$};
\draw [->,thick,dashed] (X2) -- (X7) node [midway,left] {};
\draw [->,thick,dashed] (X4) -- (X9) node [midway,left] {};
\end{tikzpicture}
\end{center}
\end{itemize}
\end{definition}

An {\em$n$-exact category} is a pair $(\mathcal{M},\mathcal{X})$ where $\mathcal{M}$ is an additive category and $\mathcal{X}$ is an
$n$-exact structure on $\mathcal{M}$. If the class $\mathcal{X}$ is clear from the context, we identify $\mathcal{M}$ with the
pair $(\mathcal{M},\mathcal{X})$. The members of $\mathcal{X}$ are called {\em$\mathcal{X}$-admissible $n$-exact sequences}, or simply
{\em admissible $n$-exact sequences} when $\mathcal{X}$ is clear from the context. Furthermore, if
\begin{equation}
X^0 \overset{d^0}{\rightarrowtail} X^1 \overset{d^1}{\rightarrow} \cdots \overset{d^{n-1}}{\rightarrow} X^n \overset{d^n}{\twoheadrightarrow} X^{n+1} \notag
\end{equation}
is an admissible $n$-exact sequence, $d^0$ is called {\em admissible monomorphism} and $d^n$ is called
{\em admissible epimorphism}.

Let $\mathcal{A}$ be an additive category and $\mathcal{B}$ be a full subcategory of $\mathcal{A}$. $\mathcal{B}$ is called
{\em covariantly finite} in $\mathcal{A}$ if for every $A\in \mathcal{A}$ there exists an object $B\in\mathcal{B}$ and a morphism
$f : A\rightarrow B$ such that, for all $B'\in\mathcal{B}$, the sequence of abelian groups $\Hom_{\mathcal{A}}(B, B')\rightarrow \Hom_{\mathcal{A}}(A, B')\rightarrow 0$ is exact. Such a morphism $f$ is called a {\em left $\mathcal{B}$-approximation} of $A$. The notions of {\em contravariantly
finite subcategory} of $\mathcal{A}$ and right {\em$\mathcal{B}$-approximation} are defined dually. A {\em functorially
finite subcategory} of $\mathcal{A}$ is a subcategory which is both covariantly and contravariantly finite
in $\mathcal{A}$ \cite[Page 113]{AR}.

\begin{definition}$($\cite[Definition 4.13]{J}$)$
Let $(\mathcal{E},\mathcal{X})$ be an exact category and $\mathcal{M}$ a subcategory of $\mathcal{E}$. $\mathcal{M}$ is called an {\em$n$-cluster tilting subcategory} of $(\mathcal{E},\mathcal{X})$ if the following conditions are satisfied.
\begin{itemize}
\item[$(i)$]
Every object $E\in \mathcal{E}$ has a left $\mathcal{M}$-approximation by an $\mathcal{X}$-admissible monomorphism $E\rightarrowtail M$.
\item[$(ii)$]
Every object $E\in \mathcal{E}$ has a right $\mathcal{M}$-approximation by an $\mathcal{X}$-admissible epimorphism $M'\twoheadrightarrow E$.
\item[$(iii)$]
We have
\begin{align}
\mathcal{M}& = \{ E\in \mathcal{E} \mid \forall i\in \{1, \ldots, n-1 \}, \Ext_{\mathcal{E}}^i(E,\mathcal{M})=0 \}\notag \\
& =\{ E\in \mathcal{E} \mid \forall i\in \{1, \ldots, n-1 \}, \Ext_{\mathcal{E}}^i(\mathcal{M},E)=0 \}.\notag
\end{align}
\end{itemize}
We call $\mathcal{M}$ an {\em$n\mathbb{Z}$-cluster tilting subcategory} of $(\mathcal{E},\mathcal{X})$ if the following additional condition is satisfied:
\begin{itemize}
\item[$(iv)$]
If $\Ext^k(\mathcal{M},\mathcal{M})\neq 0$ then $k\in n\mathbb{Z}$.
\end{itemize}
An $n$-cluster tilting subcategory of abelian category $\mathcal{E}$ is $n$-cluster tilting subcategory of the exact category $\mathcal{E}$ with the exact structure of all short exact sequences.

Note that $\mathcal{E}$ itself is the unique $1$-cluster tilting subcategory of $\mathcal{E}$.
\end{definition}

A full subcategory $\mathcal{M}$ of an exact or abelian category $\mathcal{E}$ is called {\em$n$-rigid}, if for every
two objects $M,N\in \mathcal{M}$ and for every $k\in \{1,\cdots, n - 1\}$, we have $\Ext_{\mathcal{E}}^i(\mathcal{M},\mathcal{M})=0$. Any
$n$-cluster tilting subcategory $\mathcal{M}$ of an exact category $\mathcal{E}$ is $n$-rigid.

The following theorem gives the main source of $n$-exact categories.

\begin{theorem}$($\cite[Theorem 4.14]{J}$)$\label{2.5}
Every $n$-cluster tilting subcategory of an abelian (resp. exact) category inherit natural structure of an $n$-abelian (resp. $n$-exact) category.
\end{theorem}

\begin{theorem}$($\cite{EN, Kv}$)$\label{2.6}
Every small $n$-abelian category is equivalent to an $n$-cluster tilting subcategory f an abelian category.
\end{theorem}

The first author in \cite{E} showed that there are $n$-exact categories which are not $n$-cluster tilting subcategories. However there is a nice embedding of small $n$-exact categories in abelian categories that we recall it in the next section.

%%%%%%%%%%%%%%%%%%%%%%%%%%%%%%%%%%%%%%

\section{Gabriel-Quillen embedding for $n$-exact categories}
Let $\mathcal{M}$ be a small $n$-exact category. In this section we first recall the Gabriel-Quillen embedding theorem \cite[Theorem]{E}. Then we show that it detects $n$-exact sequences and its essential image is closed under $n$-extensions, up to Yoneda equivalence. This allows us to compute the group of $n$-extensions introduced in \cite{Lu} in a Grothendieck category.

Let $\mathcal{M}$ be a small $n$-exact category. Recall that $\rm{Mod}(\mathcal{M})$ is
the category of all additive contravariant functors from $\mathcal{M}$ to the category of all abelian
groups. It is an abelian category with all limits and colimits, which are defined point-wise. Also by Yoneda's lemma, representable functors are projective and the direct sum of all representable functors $\Sigma_{X\in \mathcal{M}}\Hom(-,X)$, is a generator for $\rm{Mod}(\mathcal{M})$. Thus $\rm{Mod}(\mathcal{M})$ is a Grothendieck category \cite[Proposition 5.21]{Fr}.

A functor $F\in \rm{Mod}(\mathcal{M})$ is called {\em weakly effaceable}, if for each object $X\in \mathcal{M}$ and $x\in F(X)$
there exists an admissible epimorphism $f : Y \rightarrow X$ such that $F(f)(x) = 0$. We denote by
$\rm Eff(\mathcal{M})$ the full subcategory of all weakly effaceable functors. For each $k\in \{1,\cdots, n\}$ we denote by
$\mathcal{L}_k(\mathcal{M})$ the full subcategory of $\rm{Mod}(\mathcal{M})$ consist of all functors like $F$ such that for every $n$-exact sequence
\begin{equation}
X^0 \overset{}{\rightarrowtail} X^1 \overset{}{\rightarrow} \cdots \overset{}{\rightarrow} X^n \overset{}{\twoheadrightarrow} X^{n+1} \notag
\end{equation}
the sequence of abelian groups
\begin{equation}
0 \overset{}{\rightarrow} F(X^{n+1}) \overset{}{\rightarrow}F(X^n)\rightarrow \cdots \overset{}{\rightarrow} F(X^{n-k})\notag
\end{equation}
is exact. Also for a Serre subcategory $\mathcal{C}$ of an abelian category $\mathcal{A}$ we set $\mathcal{C}^{{\bot}_k} = \{A \in \mathcal{A} | \Ext_{\mathcal{A}}^{0,...,k}(\mathcal{C}, A) = 0\}$. In particular $\mathcal{C}^{{\bot}_1}=\mathcal{C}^{\bot}$ is the subcategory of $\mathcal{C}$-closed objects \cite[Page 39]{KrB}.
In the following proposition we collect the basic properties of the Gabriel-Quillen embedding theorem.

\begin{proposition}$($\cite[Section 3]{E}$)$\label{3.1}
\begin{itemize}
\item[(i)]
$\rm Eff(\mathcal{M})$ is a localising subcategory of $\rm{Mod}(\mathcal{M})$.
\item[(ii)]
$\rm Eff(\mathcal{M})^{\bot}=\mathcal{L}_1(\mathcal{M})$.
\item[(iii)]
For every $k\in \{1,\cdots, n\}$, $\rm Eff(\mathcal{M})^{\bot_k}=\mathcal{L}_k(\mathcal{M})$.
\end{itemize}
Denote by $H:\mathcal{M}\rightarrow \mathcal{L}_1(\mathcal{M})$ the composition of the Yoneda functor $\mathcal{M}\rightarrow \Mod(\mathcal{M})$ with the localization functor $\Mod(\mathcal{M})\rightarrow \frac{\Mod(\mathcal{M})}{\rm{Eff}(\mathcal{M})}\simeq \rm{Eff}(\mathcal{M})^{\bot}=\mathcal{L}_1(\mathcal{M})$. Thus $H(X)=(-,X):\mathcal{M}^{op}\rightarrow \rm Ab$. For simplicity we denote $(-,X)$ by $H_X$. Then
\begin{itemize}
\item[(i)]
For every $n$-exact sequence $X^0\rightarrowtail X^1\rightarrow \cdots\rightarrow X^n\twoheadrightarrow X^{n+1}$ in $\mathcal{M}$,
\begin{center}
$0\rightarrow H_{X^0}\rightarrow H_{X^1}\rightarrow \cdots\rightarrow H_{X^{n}}\rightarrow H_{X^{n+1}}\rightarrow 0$
\end{center}
is exact in $\mathcal{L}_1(\mathcal{M})$.
\item[(ii)]
The essential image of $H:\mathcal{M}\rightarrow \mathcal{L}_1(\mathcal{M})$ is $n$-rigid.
\end{itemize}
\end{proposition}

In the following proposition we prove that the canonical functor $H:\mathcal{M}\rightarrow \mathcal{L}_1(\mathcal{M})$ detects $n$-exact sequences.

\begin{proposition}\label{3.2}
Let $Y:Y^0\rightarrow Y^1\rightarrow \cdots\rightarrow Y^n\rightarrow Y^{n+1}$ be a complex of objects in $\mathcal{M}$ such that
\begin{equation}\label{det}
0\rightarrow H_{Y^0}\rightarrow H_{Y^1}\rightarrow \cdots\rightarrow H_{Y^{n}}\rightarrow H_{Y^{n+1}}\rightarrow 0
\end{equation}
is exact in $\mathcal{L}_1(\mathcal{M})$. Then $Y$ is an admissible $n$-exact sequence in $\mathcal{M}$.
\begin{proof}
Because the essential image of $H:\mathcal{M}\rightarrow \mathcal{L}_1(\mathcal{M})$ is $n$-rigid, by the similar argument as in the proof of \cite[Proposition 2.2]{JK} for each object $Z\in \mathcal{M}$ we have the following exact sequence of abelian groups.
\begin{equation}
0\rightarrow \Hom(H_Z,H_{Y^0})\rightarrow \Hom(H_Z,H_{Y^1})\rightarrow \cdots\rightarrow \Hom(H_Z,H_{Y^{n}})\rightarrow \Hom(H_Z,H_{Y^{n+1}}). \notag
\end{equation}
Thus by Yoneda's lemma $Y$ is a left $n$-exact sequence. Dually it is a right $n$-exact sequence, so it is an $n$-exact sequence. We need to show that $Y$ is an admissible $n$-exact sequence. The cokernel of $H_{Y^{n}}\rightarrow H_{Y^{n+1}}$, denoted by $C$, is weakly effaceable. In particular, there exist $X^n\in \mathcal{M}$ and an admissible epimorphism $X^n\twoheadrightarrow Y^{n+1}$ in $\mathcal{M}$, such that $C(Y^{n+1})\rightarrow C(X^n)$ carries the image of $1_{Y^{n+1}}$ to $0$. This means that
there is a commutative diagram with exact rows in $\mathcal{L}_1(\mathcal{M})$ of the following form for an admissible $n$-exact sequence $X:X^0\rightarrowtail X^1\rightarrow \cdots \rightarrow X^n\twoheadrightarrow Y^{n+1}$ in $\mathcal{M}$.
\begin{center}
\begin{tikzpicture}
%\node (X1) at (-8,1) {$0$};
%\node (X2) at (-6.5,1) {$X^0$};
\node (X3) at (-5,1) {$0$};
\node (X4) at (-3.5,1) {$H_{X^0}$};
\node (X5) at (-1.75,1) {$H_{X^1}$};
\node (X6) at (0,1) {$\cdots$};
\node (X7) at (1.5,1) {$H_{X^n}$};
\node (X8) at (3.5,1) {$H_{Y^{n+1}}$};
\node (X9) at (5,1) {$0$};
\node (X10) at (-5,-0.5) {$0$};
\node (X11) at (-3.5,-0.5) {$H_{Y^0}$};
\node (X12) at (-1.75,-0.5) {$H_{Y^1}$};
\node (X13) at (0,-0.5) {$\cdots$};
\node (X14) at (1.5,-0.5) {$H_{Y^n}$};
\node (X15) at (3.5,-0.5) {$H_{Y^{n+1}}$};
\node (X16) at (5,-0.5) {$0,$};
%\draw [->,thick] (X1) -- (X2) node [midway,left] {};
%\draw [->,thick] (X2) -- (X3) node [midway,left] {};
\draw [->,thick] (X3) -- (X4) node [midway,left] {};
\draw [->,thick] (X4) -- (X5) node [midway,left] {};
\draw [->,thick] (X5) -- (X6) node [midway,left] {};
\draw [->,thick] (X6) -- (X7) node [midway,above] {};
\draw [->,thick] (X7) -- (X8) node [midway,left] {};
\draw [->,thick] (X8) -- (X9) node [midway,above] {};
\draw [->,thick] (X10) -- (X11) node [midway,left] {};
\draw [->,thick] (X11) -- (X12) node [midway,above] {};
\draw [->,thick] (X12) -- (X13) node [midway,above] {};
\draw [->,thick] (X13) -- (X14) node [midway,above] {};
\draw [->,thick] (X14) -- (X15) node [midway,above] {};
\draw [->,thick] (X15) -- (X16) node [midway,above] {};
\draw [->,thick,dotted] (X4) -- (X11) node [midway,above] {};
\draw [->,thick,dotted] (X5) -- (X12) node [midway,above] {};
\draw [->,thick] (X7) -- (X14) node [midway,above] {};
\draw [double,-,thick] (X8) -- (X15) node [midway,above] {};
\end{tikzpicture}
\end{center}
where the dotted arrows are induced by the factorization property of $n$-kernel.
Because the top row is induced by an admissible $n$-exact sequence, by the dual of obscure axiom (\cite[Proposition 4.11 ]{J}) and Yoneda's lemma the bottom row is also induced by an admissible $n$-exact sequence.
\end{proof}
\end{proposition}

In the classical case, the Gabriel-Quillen embedding $\mathcal{E}\hookrightarrow \rm Lex(\mathcal{E})$ identifies $\mathcal{E}$ with an extension closed subcategory of the abelian category $\rm Lex(\mathcal{E})$. We prove a similar result for $n$-exact categories. First let recall some facts about Yoneda extension groups in exact categories.

Let $\mathcal{E}$ be an exact category, $k$ a positive integer and $A, C\in \mathcal{E}$. An exact sequence
\begin{equation}
 A\rightarrowtail X_{k-1}\rightarrow \cdots \rightarrow X_0\twoheadrightarrow C \notag
\end{equation}
in $\mathcal{E}$ is called a {\em $k$-fold extension of $C$ by $A$}. Two $k$-fold extensions of $C$ by $A$,
\begin{equation}
\xi: A\rightarrowtail B_{k-1}\rightarrow \cdots \rightarrow B_0\twoheadrightarrow C\notag
\end{equation}
and
\begin{equation}
\xi': A\rightarrowtail B'_{k-1}\rightarrow \cdots \rightarrow B'_0\twoheadrightarrow C \notag
\end{equation}
are said to be {\em Yoneda equivalent} if there is a chain of $k$-fold extensions of $C$ by $A$
\begin{equation}
\xi=\xi_0,  \xi_1, \ldots ,\xi_{l-1}, \xi_l=\xi' \notag
\end{equation}
such that for every $i\in \{0,\cdots,l-1\}$, we have either a chain map $\xi_i\rightarrow \xi_{i+1}$ starting with $1_A$ and ending with $1_C$, or a chain map $\xi_{i+1}\rightarrow \xi_{i}$ starting with $1_A$ and ending with $1_C$.
$\Ext_{\mathcal{E}}^k(C,A)$ is defined as the set of Yoneda equivalence classes of $k$-fold extensions of $C$ by $A$ \cite{Mac, Mi}.

\begin{remark}\label{3.3}
let $\mathcal{E}$ be an exact category and $X\rightarrowtail Y\twoheadrightarrow Z$ be a conflation in $\mathcal{E}$. The for every object $A\in \mathcal{E}$, in the long exact sequence
\begin{align*}
0&\rightarrow \Hom_{\mathcal{E}}(A,X)\rightarrow \Hom_{\mathcal{E}}(A,Y)\rightarrow \Hom_{\mathcal{E}}(A,Z)\\
&\rightarrow \Ext^1_{\mathcal{E}}(A,X)\rightarrow \Ext^1_{\mathcal{E}}(A,Y)\rightarrow \Ext^1_{\mathcal{E}}(A,Z)\\
&\rightarrow \cdots
\end{align*}
for every $i\geq 1$, the connecting morphism $\Ext^i_{\mathcal{E}}(A,Z)\rightarrow\Ext^{i+1}_{\mathcal{E}}(A,X)$ is given by splicing an $i$-fold extension
\begin{center}
$Z\rightarrowtail E^1\rightarrow E^2\rightarrow\cdots\rightarrow E^i\twoheadrightarrow A$
\end{center}
with the short exact sequence $X\rightarrowtail Y\twoheadrightarrow Z$ and obtaining the following $(i+1)$-fold extension.
\begin{center}
$X\rightarrowtail Y\rightarrow E^1\rightarrow E^2\rightarrow\cdots\rightarrow E^i\twoheadrightarrow A.$
\end{center}
\end{remark}

\begin{remark}\label{3.4}

Let $ K\rightarrowtail B_0\twoheadrightarrow C$ and $ A\rightarrowtail B_1'\twoheadrightarrow K'$ be two conflation in an exact category $\mathcal{E}$ and $\sigma:K\rightarrow K'$ be a morphism. Then by taking pullback and pushout along $\sigma$, we have the following commutative diagram with exact rows.
\begin{center}
\begin{tikzpicture}
%\node (X1) at (-4.,1) {$0$};
\node (X2) at (-2.5,1) {$A$};
\node (X3) at (-1,1) {$W_1$};
\node (X4) at (0,0) {$K$};
\node (X5) at (1,1) {$B_0$};
\node (X6) at (2.5,1) {$C$};
%\node (X7) at (4,1) {$0$};
%\node (X8) at (-4,-1) {$0$};
\node (X9) at (-2.5,-1) {$A$};
\node (X10) at (-1,-1) {$B'_1$};
\node (X11) at (0,-2) {$K'$};
\node (X12) at (1,-1) {$W_0$};
\node (X13) at (2.5,-1) {$C$};
%\node (X14) at (4,-1) {$0$};
%\draw [->,thick] (X1) -- (X2) node [midway,left] {};
\draw [>->,thick] (X2) -- (X3) node [midway,left] {};
\draw [->,thick] (X3) -- (X4) node [midway,left] {};
\draw [->,thick] (X3) -- (X5) node [midway,left] {};
\draw [>->,thick] (X4) -- (X5) node [midway,left] {};
\draw [->>,thick] (X5) -- (X6) node [midway,left] {};
%\draw [->,thick] (X6) -- (X7) node [midway,above] {};
%\draw [->,thick] (X8) -- (X9) node [midway,above] {};
\draw [>->,thick] (X9) -- (X10) node [midway,above] {};
\draw [->>,thick] (X10) -- (X11) node [midway,above] {};
\draw [->,thick] (X10) -- (X12) node [midway,left] {};
\draw [->,thick] (X11) -- (X12) node [midway,above] {};
\draw [->>,thick] (X12) -- (X13) node [midway,above] {};
%\draw [->,thick] (X13) -- (X14) node [midway,above] {};
\draw [double,-,thick] (X2) -- (X9) node [midway,above] {};
\draw [->,thick] (X3) -- (X10) node [midway,above] {};
\draw [->,thick] (X4) -- (X11) node [midway,left] {};
\draw [->,thick] (X5) -- (X12) node [midway,above] {};
\draw [double,-,thick] (X6) -- (X13) node [midway,above] {};
\end{tikzpicture}
\end{center}
These two rows are Yoneda equivalent in $\Ext_{\mathcal{E}}^2(C,A)$. Now consider the following commutative diagram with exact rows
\begin{center}
\begin{tikzpicture}
%\node (X1) at (-4.5,0.5) {$0$};
%\node (X2) at (-3,0.5) {$0$};
%\node (X3) at (-1.5,0.5) {$X^0$};
%\node (X4) at (0,0.5) {$0$};
\node (X5) at (1.5,0.5) {$X^0$};
\node (X6) at (3,0.5) {$X^1$};
\node (X7) at (4.5,0.5) {$X^2$};
\node (X8) at (6,0.5) {$X^3$};
%\node (X9) at (7.5,0.5) {$0$};
%\node (X10) at (-5,-1) {$0$};
%\node (X11) at (0,-1) {$\xi:0$};
\node (X12) at (1.5,-1) {$\xi:Y^0$};
\node (X13) at (3,-1) {$Y^1$};
\node (X14) at (4.5,-1) {$Y^2$};
\node (X15) at (6,-1) {$X^3.$};
%\node (X16) at (7.5,-1) {$0.$};
%\draw [->,thick] (X1) -- (X2) node [midway,left] {};
%\draw [->,thick] (X2) -- (X3) node [midway,left] {};
%\draw [->,thick] (X3) -- (X4) node [midway,left] {};
%\draw [->,thick] (X4) -- (X5) node [midway,left] {};
\draw [>->,thick] (X5) -- (X6) node [midway,left] {};
\draw [->,thick] (X6) -- (X7) node [midway,above] {};
\draw [->>,thick] (X7) -- (X8) node [midway,above] {$f^2$};
%\draw [->,thick] (X8) -- (X9) node [midway,above] {};
%\draw [->,thick] (X10) -- (X11) node [midway,left] {};
%\draw [->,thick] (X11) -- (X12) node [midway,above] {};
\draw [>->,thick] (X12) -- (X13) node [midway,above] {};
\draw [->,thick] (X13) -- (X14) node [midway,above] {};
\draw [->>,thick] (X14) -- (X15) node [midway,above] {$g^2$};
%\draw [->,thick] (X15) -- (X16) node [midway,above] {};
\draw [->,thick] (X7) -- (X14) node [midway,above] {};
\draw [double,-,thick] (X8) -- (X15) node [midway,above] {};
\end{tikzpicture}
\end{center}
By taking a pullback
\begin{center}
\begin{tikzpicture}
\node (X1) at (0,0) {$\Ker (g^2)$};
\node (X2) at (0,1.5) {$\Ker (f^2)$};
\node (X3) at (-2,0) {$Y^1$};
\node (X4) at (-2,1.5) {$Y_{pb}^{1}$};
\draw [->,thick] (X2) -- (X1) node [midway,left] {};
\draw [->,thick] (X3) -- (X1) node [midway,left] {};
\draw [->,dashed] (X4) -- (X2) node [midway,right] {};
\draw [->,dashed] (X4) -- (X3) node [midway,left] {};
\end{tikzpicture}
\end{center}
we obtain the following commutative diagram with exact rows, where the bottom row is Yoneda equivalent to $\xi$.
\begin{center}
\begin{tikzpicture}
%\node (X1) at (-4.5,0.5) {$0$};
%\node (X2) at (-3,0.5) {$0$};
%\node (X3) at (-1.5,0.5) {$X^0$};
%\node (X4) at (0,0.5) {$0$};
\node (X5) at (1.5,0.5) {$X^0$};
\node (X6) at (3,0.5) {$X^1$};
\node (X7) at (4.5,0.5) {$X^2$};
\node (X8) at (6,0.5) {$X^3$};
%\node (X9) at (7.5,0.5) {$0$};
%\node (X10) at (-5,-1) {$0$};
%\node (X11) at (0,-1) {$0$};
\node (X12) at (1.5,-1) {$Y^0$};
\node (X13) at (3,-1) {$Y_{pb}^1$};
\node (X14) at (4.5,-1) {$X^2$};
\node (X15) at (6,-1) {$X^3.$};
%\node (X16) at (7.5,-1) {$0.$};
%\draw [->,thick] (X1) -- (X2) node [midway,left] {};
%\draw [->,thick] (X2) -- (X3) node [midway,left] {};
%\draw [->,thick] (X3) -- (X4) node [midway,left] {};
%\draw [->,thick] (X4) -- (X5) node [midway,left] {};
\draw [>->,thick] (X5) -- (X6) node [midway,left] {};
\draw [->,thick] (X6) -- (X7) node [midway,above] {};
\draw [->>,thick] (X7) -- (X8) node [midway,above] {};
%\draw [->,thick] (X8) -- (X9) node [midway,above] {};
%\draw [->,thick] (X10) -- (X11) node [midway,left] {};
%\draw [->,thick] (X11) -- (X12) node [midway,above] {};
\draw [>->,thick] (X12) -- (X13) node [midway,above] {};
\draw [->,thick] (X13) -- (X14) node [midway,above] {};
\draw [->>,thick] (X14) -- (X15) node [midway,above] {};
%\draw [->,thick] (X15) -- (X16) node [midway,above] {};
\draw [double,-,thick] (X7) -- (X14) node [midway,above] {};
\draw [double,-,thick] (X8) -- (X15) node [midway,above] {};
\end{tikzpicture}
\end{center}
In the proof of the following proposition we use this argument several times.
\end{remark}

\begin{theorem}\label{3.5}
Let $\mathcal{M}$ be a small $n$-exact category. The embedding $\mathcal{M}\hookrightarrow \mathcal{L}_1(\mathcal{M})$ is closed under $n$-extensions, up to Yoneda equivalence.
\begin{proof} Let $k\in \{1,\ldots,n-1\}$, $E^0, E^{n+1}\in \mathcal{L}_1(\mathcal{M})$ and
\begin{center}
$\xi: 0\rightarrow E^0\rightarrow E^1\rightarrow \cdots \rightarrow E^n\rightarrow E^{n+1}\rightarrow 0$
\end{center}
be a $k$-fold extension of $E^{n+1}$ by $E^0$ in $\mathcal{L}_1(\mathcal{M})$. There exist $Y,X^{n+1}\in \mathcal{M}$ such that $H(Y)=E^0$ and $H(X^{n+1})=E^{n+1}$. Then the cokernel of $E^n\rightarrow H(X^{n+1})$, denoted by $C$, is effaceable. In particular, there exist $X^n\in \mathcal{M}$ and an admissible epimorphism $X^n\twoheadrightarrow X^{n+1}$ in $\mathcal{M}$, such that $C(X^{n+1})\rightarrow C(X^n)$ carries the image of $1_{X^{n+1}}$ to $0$. This means that
there is a commutative diagram with exact rows in $\mathcal{L}_1(\mathcal{M})$ of the following form for an admissible $n$-exact sequence $X:X^0\rightarrowtail X^1\rightarrow \cdots \rightarrow X^n\twoheadrightarrow X^{n+1}$ in $\mathcal{M}$.
\begin{center}
\begin{tikzpicture}
%\node (X1) at (-8,1) {$0$};
%\node (X2) at (-6.5,1) {$X^0$};
\node (X3) at (-5,1) {$0$};
\node (X4) at (-3.5,1) {$H(X^0)$};
\node (X5) at (-1.5,1) {$H(X^1)$};
\node (X6) at (.5,1) {$\cdots$};
\node (X7) at (2.5,1) {$H(X^n)$};
\node (X8) at (5,1) {$H(X^{n+1})$};
\node (X9) at (7,1) {$0$};
\node (X10) at (-5,-0.5) {$0$};
\node (X11) at (-3.5,-0.5) {$H(Y)$};
\node (X12) at (-1.5,-0.5) {$E^1$};
\node (X13) at (.5,-0.5) {$\cdots$};
\node (X14) at (2.5,-0.5) {$E^n$};
\node (X15) at (5,-0.5) {$H(X^{n+1})$};
\node (X16) at (7,-0.5) {$0.$};
%\draw [->,thick] (X1) -- (X2) node [midway,left] {};
%\draw [->,thick] (X2) -- (X3) node [midway,left] {};
\draw [->,thick] (X3) -- (X4) node [midway,left] {};
\draw [->,thick] (X4) -- (X5) node [midway,above] {$f^0$};
\draw [->,thick] (X5) -- (X6) node [midway,above] {$f^1$};
\draw [->,thick] (X6) -- (X7) node [midway,above] {$f^{n-1}$};
\draw [->,thick] (X7) -- (X8) node [midway,above] {$f^n$};
\draw [->,thick] (X8) -- (X9) node [midway,above] {};
\draw [->,thick] (X10) -- (X11) node [midway,above] {};
\draw [->,thick] (X11) -- (X12) node [midway,above] {$g^0$};
\draw [->,thick] (X12) -- (X13) node [midway,above] {$g^1$};
\draw [->,thick] (X13) -- (X14) node [midway,above] {$g^{n-1}$};
\draw [->,thick] (X14) -- (X15) node [midway,above] {$g^n$};
\draw [->,thick] (X15) -- (X16) node [midway,above] {};
\draw [->,thick] (X7) -- (X14) node [midway,above] {};
\draw [double,-,thick] (X8) -- (X15) node [midway,above] {};
\end{tikzpicture}
\end{center}
By taking pullback along $\Ker(f^n)\rightarrow \Ker(g^n)$ and by Remark \ref{3.4}, we obtain the following commutative diagram with exact rows, such that the bottom row is Yoneda equivalent to $\xi$.
\begin{center}
\begin{tikzpicture}
\node (X1) at (-4.5,0.5) {$0$};
\node (X2) at (-3,0.5) {$H(X^0)$};
\node (X3) at (-1,0.5) {$H(X^1)$};
\node (X4) at (.5,0.5) {$\cdots$};
\node (X5) at (2.5,0.5) {$H(X^{n-2})$};
\node (X6) at (5,0.5) {$H(X^{n-1})$};
\node (X7) at (7.5,0.5) {$H(X^n)$};
\node (X8) at (10,0.5) {$H(X^{n+1})$};
\node (X9) at (11.5,0.5) {$0$};
\node (X10) at (-4.5,-1) {$0$};
\node (X11) at (-3,-1) {$H(Y)$};
\node (X12) at (-1,-1) {$E^1$};
\node (X13) at (.5,-1) {$\cdots$};
\node (X14) at (2.5,-1) {$E^{n-2}$};
\node (X15) at (5,-1) {$E_{pb}^{n-1}$};
\node (X16) at (7.5,-1) {$H(X^n)$};
\node (X17) at (10,-1) {$H(X^{n+1})$};
\node (X18) at (11.5,-1) {$0$};
\draw [->,thick] (X1) -- (X2) node [midway,left] {};
\draw [->,thick] (X2) -- (X3) node [midway,left] {};
\draw [->,thick] (X3) -- (X4) node [midway,left] {};
\draw [->,thick] (X4) -- (X5) node [midway,left] {};
\draw [->,thick] (X5) -- (X6) node [midway,left] {};
\draw [->,thick] (X6) -- (X7) node [midway,above] {};
\draw [->,thick] (X7) -- (X8) node [midway,left] {};
\draw [->,thick] (X8) -- (X9) node [midway,above] {};
\draw [->,thick] (X10) -- (X11) node [midway,left] {};
\draw [->,thick] (X11) -- (X12) node [midway,above] {};
\draw [->,thick] (X12) -- (X13) node [midway,above] {};
\draw [->,thick] (X13) -- (X14) node [midway,above] {};
\draw [->,thick] (X14) -- (X15) node [midway,above] {};
\draw [->,thick] (X15) -- (X16) node [midway,above] {};
\draw [->,thick] (X16) -- (X17) node [midway,above] {};
\draw [->,thick] (X17) -- (X18) node [midway,above] {};
\draw [double,-,thick] (X7) -- (X16) node [midway,above] {};
\draw [double,-,thick] (X8) -- (X17) node [midway,above] {};
\end{tikzpicture}
\end{center}
Now for $j\in\{1,\cdots,n-1\}$ we set $C^j:=\Imm(H(X^j)\rightarrow H(X^{j+1}))$ in $\mathcal{L}_1(\mathcal{M})$. By applying the functor $\Hom_{\mathcal{L}_1(\mathcal{M})}(-,H(Y))$ to the short exact sequence
\begin{equation}
\eta_{n-1}:0\rightarrow C^{n-2}\rightarrow H(X^{n-1})\rightarrow C^{n-1}\rightarrow 0\notag
\end{equation}
we obtain the following exact sequence.

$
\Ext_{\mathcal{L}_1(\mathcal{M})}^{n-2}(H(X^{n-1}),H(Y))\rightarrow \Ext_{\mathcal{L}_1(\mathcal{M})}^{n-2}(C^{n-2},H(Y))\overset{\alpha}{\rightarrow}\Ext_{\mathcal{L}_1(\mathcal{M})}^{n-1}(C^{n-1},H(Y))\\
\rightarrow \Ext_{\mathcal{L}_1(\mathcal{M})}^{n-1}(H(X^{n-1}),H(Y))=0. \notag
$

Therefor $\alpha$ is a surjective map, and so by Remark \ref{3.3} there is an exact sequence
\begin{equation}
\epsilon:0\rightarrow H(Y)\rightarrow F^1\rightarrow F^2\rightarrow \cdots\rightarrow F^{n-2}\rightarrow C^{n-2}\rightarrow 0 \notag
\end{equation}
such that $\epsilon \eta_{n-1}$ is Yoneda equivalent to
\begin{equation}
0\rightarrow H(Y)\rightarrow E^1\rightarrow E^2\rightarrow \cdots\rightarrow E^{n-2}\rightarrow E_{pb}^{n-1}\rightarrow C^{n-1}\rightarrow 0 \notag
\end{equation}
Thus we have the following diagram, where the bottom row is still Yoneda equivalent to $\xi$.
\begin{center}
\begin{tikzpicture}
\node (X1) at (-4.5,0.5) {$0$};
\node (X2) at (-3,0.5) {$H(X^0)$};
\node (X3) at (-1,0.5) {$H(X^1)$};
\node (X4) at (.5,0.5) {$\cdots$};
\node (X5) at (2.5,0.5) {$H(X^{n-2})$};
\node (X6) at (5,0.5) {$H(X^{n-1})$};
\node (X7) at (7.5,0.5) {$H(X^n)$};
\node (X8) at (10,0.5) {$H(X^{n+1})$};
\node (X9) at (11.5,0.5) {$0$};
\node (X10) at (-4.5,-1) {$0$};
\node (X11) at (-3,-1) {$H(Y)$};
\node (X12) at (-1,-1) {$F^1$};
\node (X13) at (.5,-1) {$\cdots$};
\node (X14) at (2.5,-1) {$F^{n-2}$};
\node (X15) at (5,-1) {$H(X^{n-1})$};
\node (X16) at (7.5,-1) {$H(X^n)$};
\node (X17) at (10,-1) {$H(X^{n+1})$};
\node (X18) at (11.5,-1) {$0$};
\draw [->,thick] (X1) -- (X2) node [midway,left] {};
\draw [->,thick] (X2) -- (X3) node [midway,left] {};
\draw [->,thick] (X3) -- (X4) node [midway,left] {};
\draw [->,thick] (X4) -- (X5) node [midway,left] {};
\draw [->,thick] (X5) -- (X6) node [midway,left] {};
\draw [->,thick] (X6) -- (X7) node [midway,above] {};
\draw [->,thick] (X7) -- (X8) node [midway,left] {};
\draw [->,thick] (X8) -- (X9) node [midway,above] {};
\draw [->,thick] (X10) -- (X11) node [midway,left] {};
\draw [->,thick] (X11) -- (X12) node [midway,above] {};
\draw [->,thick] (X12) -- (X13) node [midway,above] {};
\draw [->,thick] (X13) -- (X14) node [midway,above] {};
\draw [->,thick] (X14) -- (X15) node [midway,above] {};
\draw [->,thick] (X15) -- (X16) node [midway,above] {};
\draw [->,thick] (X16) -- (X17) node [midway,above] {};
\draw [->,thick] (X17) -- (X18) node [midway,above] {};
\draw [double,-,thick] (X6) -- (X15) node [midway,above] {};
\draw [double,-,thick] (X7) -- (X16) node [midway,above] {};
\draw [double,-,thick] (X8) -- (X17) node [midway,above] {};
\end{tikzpicture}
\end{center}
By repeating this argument we obtain the following commutative diagram where the bottom row is Yoneda equivalent to $\xi$.
\begin{center}
\begin{tikzpicture}
\node (X1) at (-4.5,0.5) {$0$};
\node (X2) at (-3,0.5) {$H(X^0)$};
\node (X3) at (-1,0.5) {$H(X^1)$};
\node (X4) at (1,0.5) {$H(X^2)$};
\node (X5) at (2.5,0.5) {$\cdots$};
\node (X6) at (4.5,0.5) {$H(X^{n-1})$};
\node (X7) at (7,0.5) {$H(X^n)$};
\node (X8) at (9.5,0.5) {$H(X^{n+1})$};
\node (X9) at (11,0.5) {$0$};
\node (X10) at (-4.5,-1) {$0$};
\node (X11) at (-3,-1) {$H(Y)$};
\node (X12) at (-1,-1) {$G^1$};
\node (X13) at (1,-1) {$H(X^2)$};
\node (X14) at (2.5,-1) {$\cdots$};
\node (X15) at (4.5,-1) {$H(X^{n-1})$};
\node (X16) at (7,-1) {$H(X^n)$};
\node (X17) at (9.5,-1) {$H(X^{n+1})$};
\node (X18) at (11,-1) {$0$};
\draw [->,thick] (X1) -- (X2) node [midway,left] {};
\draw [->,thick] (X2) -- (X3) node [midway,left] {};
\draw [->,thick] (X3) -- (X4) node [midway,left] {};
\draw [->,thick] (X4) -- (X5) node [midway,left] {};
\draw [->,thick] (X5) -- (X6) node [midway,left] {};
\draw [->,thick] (X6) -- (X7) node [midway,above] {};
\draw [->,thick] (X7) -- (X8) node [midway,left] {};
\draw [->,thick] (X8) -- (X9) node [midway,above] {};
\draw [->,thick] (X10) -- (X11) node [midway,left] {};
\draw [->,thick] (X11) -- (X12) node [midway,above] {};
\draw [->,thick] (X12) -- (X13) node [midway,above] {};
\draw [->,thick] (X13) -- (X14) node [midway,above] {};
\draw [->,thick] (X14) -- (X15) node [midway,above] {};
\draw [->,thick] (X15) -- (X16) node [midway,above] {};
\draw [->,thick] (X16) -- (X17) node [midway,above] {};
\draw [->,thick] (X17) -- (X18) node [midway,above] {};
\draw [double,-,thick] (X4) -- (X13) node [midway,above] {};
\draw [double,-,thick] (X6) -- (X15) node [midway,above] {};
\draw [double,-,thick] (X7) -- (X16) node [midway,above] {};
\draw [double,-,thick] (X8) -- (X17) node [midway,above] {};
\end{tikzpicture}
\end{center}
Now by applying $\Hom_{\mathcal{L}_1(\mathcal{M})}(H(X^1),-)$ to the short exact sequence $0\rightarrow H(Y)\rightarrow G^1\rightarrow C^1\rightarrow 0$ we have the following exact sequence of abelian groups.
\begin{equation}
\Hom_{\mathcal{L}_1(\mathcal{M})}(H(X^1),G^1)\rightarrow \Hom_{\mathcal{L}_1(\mathcal{M})}(H(X^1),C^1)\rightarrow\Ext_{\mathcal{L}_1(\mathcal{M})}^1(H(X^1),H(Y))=0. \notag
\end{equation}
So there is a morphism $h^1:H(X^1)\rightarrow G^1$ that make the following diagram commutative. By the universal property of kernel there is also a morphism $h^0:H(X^0)\rightarrow H(Y)$ that make the following diagram commutative.
\begin{center}
\begin{tikzpicture}
\node (X1) at (-4.5,0.5) {$0$};
\node (X2) at (-3,0.5) {$H(X^0)$};
\node (X3) at (-1,0.5) {$H(X^1)$};
\node (X4) at (1,0.5) {$H(X^2)$};
\node (X5) at (2.5,0.5) {$\cdots$};
\node (X6) at (4.5,0.5) {$H(X^{n-1})$};
\node (X7) at (7,0.5) {$H(X^n)$};
\node (X8) at (9.5,0.5) {$H(X^{n+1})$};
\node (X9) at (11.5,0.5) {$0$};
\node (X10) at (-4.5,-1) {$0$};
\node (X11) at (-3,-1) {$H(Y)$};
\node (X12) at (-1,-1) {$G^1$};
\node (X13) at (1,-1) {$H(X^2)$};
\node (X14) at (2.5,-1) {$\cdots$};
\node (X15) at (4.5,-1) {$H(X^{n-1})$};
\node (X16) at (7,-1) {$H(X^n)$};
\node (X17) at (9.5,-1) {$H(X^{n+1})$};
\node (X18) at (11.5,-1) {$0.$};
\draw [->,thick] (X1) -- (X2) node [midway,left] {};
\draw [->,thick] (X2) -- (X3) node [midway,left] {};
\draw [->,thick] (X3) -- (X4) node [midway,left] {};
\draw [->,thick] (X4) -- (X5) node [midway,left] {};
\draw [->,thick] (X5) -- (X6) node [midway,left] {};
\draw [->,thick] (X6) -- (X7) node [midway,above] {};
\draw [->,thick] (X7) -- (X8) node [midway,left] {};
\draw [->,thick] (X8) -- (X9) node [midway,above] {};
\draw [->,thick] (X10) -- (X11) node [midway,left] {};
\draw [->,thick] (X11) -- (X12) node [midway,above] {};
\draw [->,thick] (X12) -- (X13) node [midway,above] {};
\draw [->,thick] (X13) -- (X14) node [midway,above] {};
\draw [->,thick] (X14) -- (X15) node [midway,above] {};
\draw [->,thick] (X15) -- (X16) node [midway,above] {};
\draw [->,thick] (X16) -- (X17) node [midway,above] {};
\draw [->,thick] (X17) -- (X18) node [midway,above] {};
\draw [->,thick] (X2) -- (X11) node [midway,left] {$h^0$};
\draw [->,thick] (X3) -- (X12) node [midway,left] {$h^1$};
\draw [double,-,thick] (X4) -- (X13) node [midway,above] {};
\draw [double,-,thick] (X6) -- (X15) node [midway,above] {};
\draw [double,-,thick] (X7) -- (X16) node [midway,above] {};
\draw [double,-,thick] (X8) -- (X17) node [midway,above] {};
\end{tikzpicture}
\end{center}
Let
\begin{center}
\begin{tikzpicture}
%\node (X1) at (-4.5,0.5) {$0$};
\node (X2) at (-3,0.5) {$X^0$};
\node (X3) at (-1.5,0.5) {$X^1$};
\node (X4) at (0,0.5) {$X^2$};
\node (X5) at (1.5,0.5) {$\cdots$};
\node (X6) at (3,0.5) {$X^{n-1}$};
\node (X7) at (4.5,0.5) {$X^n$};
\node (X8) at (6,0.5) {$X^{n+1}$};
%\node (X9) at (7.5,0.5) {$0$};
%\node (X10) at (-4.5,-1) {$0$};
\node (X11) at (-3,-1) {$Y$};
\node (X12) at (-1.5,-1) {$Y^1$};
\node (X13) at (0,-1) {$Y^2$};
\node (X14) at (1.5,-1) {$\cdots$};
\node (X15) at (3,-1) {$Y^{n-1}$};
\node (X16) at (4.5,-1) {$Y^n$};
\node (X17) at (6,-1) {$X^{n+1}$};
%\node (X18) at (7.5,-1) {$0$};
%\draw [->,thick] (X1) -- (X2) node [midway,left] {};
\draw [->,thick] (X2) -- (X3) node [midway,left] {};
\draw [->,thick] (X3) -- (X4) node [midway,left] {};
\draw [->,thick] (X4) -- (X5) node [midway,left] {};
\draw [->,thick] (X5) -- (X6) node [midway,left] {};
\draw [->,thick] (X6) -- (X7) node [midway,above] {};
\draw [->,thick] (X7) -- (X8) node [midway,left] {};
%\draw [->,thick] (X8) -- (X9) node [midway,above] {};
%\draw [->,thick] (X10) -- (X11) node [midway,left] {};
\draw [->,thick] (X11) -- (X12) node [midway,above] {};
\draw [->,thick] (X12) -- (X13) node [midway,above] {};
\draw [->,thick] (X13) -- (X14) node [midway,above] {};
\draw [->,thick] (X14) -- (X15) node [midway,above] {};
\draw [->,thick] (X15) -- (X16) node [midway,above] {};
\draw [->,thick] (X16) -- (X17) node [midway,above] {};
%\draw [->,thick] (X17) -- (X18) node [midway,above] {};
\draw [->,thick] (X2) -- (X11) node [midway,left] {$h$};
\draw [->,thick] (X3) -- (X12) node [midway,left] {};
\draw [->,thick] (X4) -- (X13) node [midway,above] {};
\draw [->,thick] (X6) -- (X15) node [midway,above] {};
\draw [->,thick] (X7) -- (X16) node [midway,above] {};
\draw [double,-,thick] (X8) -- (X17) node [midway,above] {};
\end{tikzpicture}
\end{center}
be the $n$-pushout diagram along $h$, where $H(h)=h^0$. Then applying the functor $H$, by \cite[Proposition 5.1]{Mac} we see that $\xi$ is Yoneda equivalent to
\begin{center}
$0\rightarrow H(Y)\rightarrow H(Y^1)\rightarrow \cdots \rightarrow H(Y^n)\rightarrow H(Y^{n+1})\rightarrow 0$.
\end{center}
\end{proof}
\end{theorem}

\begin{proposition}\label{3.6}
Let $\mathcal{M}$ be a small $n$-exact category.
If two $n$-fold extension
\begin{equation}
\xi:0\rightarrow H_{X^0}\rightarrow H_{X^1}\rightarrow \cdots\rightarrow H_{X^{n}}\rightarrow H_{X^{n+1}}\rightarrow 0\notag
\end{equation}
and
\begin{equation}
\eta:0\rightarrow H_{X^0}\rightarrow H_{Y^1}\rightarrow \cdots\rightarrow H_{Y^{n}}\rightarrow H_{X^{n+1}}\rightarrow 0\notag
\end{equation}
be Yoneda equivalent in $\mathcal{L}_1(\mathcal{M})$, then $ X:X^0\rightarrow X^1\rightarrow \cdots\rightarrow X^{n}\rightarrow X^{n+1}$ and $Y:X^0\rightarrow Y^1\rightarrow \cdots\rightarrow Y^{n}\rightarrow X^{n+1}$ are homotopy equivalent $n$-exact sequences in $\mathcal{M}$.
\begin{proof}
By Proposition \ref{3.2}, $X$ and $Y$ are $n$-exact sequences. The rest of proof is similar to the proof of \cite[Proposition A.1]{I1}.
Since $H:\mathcal{M}\rightarrow \mathcal{L}_1(\mathcal{M})$ is full and faithful, we identify $\mathcal{M}$ with the essential image of $H$.
Because $\mathcal{M}$ is an $n$-rigid subcategory of $\mathcal{L}_1(\mathcal{M})$ by \cite[Proposition 2.2]{JK} we have the exact sequences
\begin{equation}
0\rightarrow (-,X^0)\rightarrow (-,X^1)\rightarrow \cdots\rightarrow (-,X^{n})\rightarrow (-,X^{n+1})\overset{\alpha}{\rightarrow} \Ext_{\mathcal{L}_1(\mathcal{M})}^n(-,X^0)\notag
\end{equation}
and
\begin{equation}
0\rightarrow (-,X^0)\rightarrow (-,Y^1)\rightarrow \cdots\rightarrow (-,Y^{n})\rightarrow (-,X^{n+1})\overset{\alpha}{\rightarrow} \Ext_{\mathcal{L}_1(\mathcal{M})}^n(-,X^0).\notag
\end{equation}
Since $X$ and $Y$ are Yoneda equivalent, the image of two functor $\alpha$ and $\beta$ are the same. So we have the following commutative diagram, where dotted arrows are induced by the factorization property of $n$-kernel.
\begin{center}
\begin{tikzpicture}
\node (X3) at (-6.5,1) {$0$};
\node (X4) at (-4.5,1) {$(-,X^0)$};
\node (X5) at (-2,1) {$(-,X^1)$};
\node (X6) at (0,1) {$\cdots$};
\node (X7) at (2,1) {$(-,X^{n})$};
\node (X8) at (4.5,1) {$(-,X^{n+1})$};
\node (X9) at (8,1) {$\Ext_{\mathcal{L}_1(\mathcal{M})}^n(-,X^0)$};
\node (X10) at (-6.5,-0.5) {$0$};
\node (X11) at (-4.5,-0.5) {$(-,X^0)$};
\node (X12) at (-2,-0.5) {$(-,Y^1)$};
\node (X13) at (0,-0.5) {$\cdots$};
\node (X14) at (2,-0.5) {$(-,Y^n)$};
\node (X15) at (4.5,-0.5) {$(-,X^{n+1})$};
\node (X16) at (8,-0.5) {$\Ext_{\mathcal{L}_1(\mathcal{M})}^n(-,X^0)$};
\draw [->,thick] (X3) -- (X4) node [midway,left] {};
\draw [->,thick] (X4) -- (X5) node [midway,above] {};
\draw [->,thick] (X5) -- (X6) node [midway,above] {};
\draw [->,thick] (X6) -- (X7) node [midway,above] {};
\draw [->,thick] (X7) -- (X8) node [midway,above] {};
\draw [->,thick] (X8) -- (X9) node [midway,above] {};
\draw [->,thick] (X10) -- (X11) node [midway,above] {};
\draw [->,thick] (X11) -- (X12) node [midway,above] {};
\draw [->,thick] (X12) -- (X13) node [midway,above] {};
\draw [->,thick] (X13) -- (X14) node [midway,above] {};
\draw [->,thick] (X14) -- (X15) node [midway,above] {};
\draw [->,thick] (X15) -- (X16) node [midway,above] {};
\draw [->,thick,dotted] (X4) -- (X11) node [midway,above] {};
\draw [->,thick,dotted] (X5) -- (X12) node [midway,above] {};
\draw [->,thick,dotted] (X7) -- (X14) node [midway,above] {};
\draw [double,-,thick] (X8) -- (X15) node [midway,above] {};
\draw [double,-,thick] (X9) -- (X16) node [midway,above] {};
\end{tikzpicture}
\end{center}
By Yoneda's lemma we obtain the following morphism of $n$-exact sequences.
\begin{center}
\begin{tikzpicture}
%\node (X1) at (-4.5,0.5) {$0$};
\node (X2) at (-3,0.5) {$X^0$};
\node (X3) at (-1.5,0.5) {$X^1$};
\node (X4) at (0,0.5) {$X^2$};
\node (X5) at (1.5,0.5) {$\cdots$};
\node (X6) at (3,0.5) {$X^{n-1}$};
\node (X7) at (4.5,0.5) {$X^n$};
\node (X8) at (6,0.5) {$X^{n+1}$};
%\node (X9) at (7.5,0.5) {$0$};
%\node (X10) at (-4.5,-1) {$0$};
\node (X11) at (-3,-1) {$X^0$};
\node (X12) at (-1.5,-1) {$Y^1$};
\node (X13) at (0,-1) {$Y^2$};
\node (X14) at (1.5,-1) {$\cdots$};
\node (X15) at (3,-1) {$Y^{n-1}$};
\node (X16) at (4.5,-1) {$Y^n$};
\node (X17) at (6,-1) {$X^{n+1}$};
%\node (X18) at (7.5,-1) {$0$};
%\draw [->,thick] (X1) -- (X2) node [midway,left] {};
\draw [->,thick] (X2) -- (X3) node [midway,left] {};
\draw [->,thick] (X3) -- (X4) node [midway,left] {};
\draw [->,thick] (X4) -- (X5) node [midway,left] {};
\draw [->,thick] (X5) -- (X6) node [midway,left] {};
\draw [->,thick] (X6) -- (X7) node [midway,above] {};
\draw [->,thick] (X7) -- (X8) node [midway,left] {};
%\draw [->,thick] (X8) -- (X9) node [midway,above] {};
%\draw [->,thick] (X10) -- (X11) node [midway,left] {};
\draw [->,thick] (X11) -- (X12) node [midway,above] {};
\draw [->,thick] (X12) -- (X13) node [midway,above] {};
\draw [->,thick] (X13) -- (X14) node [midway,above] {};
\draw [->,thick] (X14) -- (X15) node [midway,above] {};
\draw [->,thick] (X15) -- (X16) node [midway,above] {};
\draw [->,thick] (X16) -- (X17) node [midway,above] {};
%\draw [->,thick] (X17) -- (X18) node [midway,above] {};
\draw [->,thick,dotted] (X2) -- (X11) node [midway,left] {$h^0$};
\draw [->,thick,dotted] (X3) -- (X12) node [midway,left] {$h^1$};
\draw [->,thick,dotted] (X4) -- (X13) node [midway,left] {$h^2$};
\draw [->,thick,dotted] (X6) -- (X15) node [midway,left] {$h^{n-1}$};
\draw [->,thick,dotted] (X7) -- (X16) node [midway,left] {$h^n$};
\draw [double,-,thick] (X8) -- (X17) node [midway,above] {};
\end{tikzpicture}
\end{center}
Dually we have the following commutative diagram.
\begin{center}
\begin{tikzpicture}
\node (X7) at (-2.5,1) {$(X^{1},-)$};
\node (X8) at (0,1) {$(X^{0},-)$};
\node (X9) at (3.5,1) {$\Ext_{\mathcal{L}_1(\mathcal{M})}^n(X^{n+1},-)$};
\node (X14) at (-2.5,-0.5) {$(Y^1,-)$};
\node (X15) at (0,-0.5) {$(X^{0},-)$};
\node (X16) at (3.5,-0.5) {$\Ext_{\mathcal{L}_1(\mathcal{M})}^n(X^{n+1},-)$};
\draw [->,thick] (X7) -- (X8) node [midway,above] {$f^0$};
\draw [->,thick] (X8) -- (X9) node [midway,above] {$\delta$};
\draw [->,thick] (X14) -- (X15) node [midway,above] {$g^0$};
\draw [->,thick] (X15) -- (X16) node [midway,above] {$\delta'$};
\draw [->,thick] (X14) -- (X7) node [midway,left] {$h^1$};
\draw [->,thick] (X15) -- (X8) node [midway,left] {$h^0$};
\draw [double,-,thick] (X9) -- (X16) node [midway,above] {};
\end{tikzpicture}
\end{center}
By assumption $\delta(1_{X^0})=\delta'(1_{X^0})$ and so $\delta(h^0-1_{X^0})=0$. Hence there is a morphism $s:X^1\rightarrow X^0$ such that $h^0=1-sf^0$. Then we have the following equivalence of $n$-exact sequences with $\tilde{h^1}= h^1-g^0s$.
\begin{center}
\begin{tikzpicture}
%\node (X1) at (-4.5,0.5) {$0$};
\node (X2) at (-3,0.5) {$X^0$};
\node (X3) at (-1.5,0.5) {$X^1$};
\node (X4) at (0,0.5) {$X^2$};
\node (X5) at (1.5,0.5) {$\cdots$};
\node (X6) at (3,0.5) {$X^{n-1}$};
\node (X7) at (4.5,0.5) {$X^n$};
\node (X8) at (6,0.5) {$X^{n+1}$};
%\node (X9) at (7.5,0.5) {$0$};
%\node (X10) at (-4.5,-1) {$0$};
\node (X11) at (-3,-1) {$X^0$};
\node (X12) at (-1.5,-1) {$Y^1$};
\node (X13) at (0,-1) {$Y^2$};
\node (X14) at (1.5,-1) {$\cdots$};
\node (X15) at (3,-1) {$Y^{n-1}$};
\node (X16) at (4.5,-1) {$Y^n$};
\node (X17) at (6,-1) {$X^{n+1}$};
%\node (X18) at (7.5,-1) {$0$};
%\draw [->,thick] (X1) -- (X2) node [midway,left] {};
\draw [->,thick] (X2) -- (X3) node [midway,left] {};
\draw [->,thick] (X3) -- (X4) node [midway,left] {};
\draw [->,thick] (X4) -- (X5) node [midway,left] {};
\draw [->,thick] (X5) -- (X6) node [midway,left] {};
\draw [->,thick] (X6) -- (X7) node [midway,above] {};
\draw [->,thick] (X7) -- (X8) node [midway,left] {};
%\draw [->,thick] (X8) -- (X9) node [midway,above] {};
%\draw [->,thick] (X10) -- (X11) node [midway,left] {};
\draw [->,thick] (X11) -- (X12) node [midway,above] {};
\draw [->,thick] (X12) -- (X13) node [midway,above] {};
\draw [->,thick] (X13) -- (X14) node [midway,above] {};
\draw [->,thick] (X14) -- (X15) node [midway,above] {};
\draw [->,thick] (X15) -- (X16) node [midway,above] {};
\draw [->,thick] (X16) -- (X17) node [midway,above] {};
%\draw [->,thick] (X17) -- (X18) node [midway,above] {};
\draw [double,-,thick] (X2) -- (X11) node [midway,left] {};
\draw [->,thick] (X3) -- (X12) node [midway,left] {$\tilde{h^1}$};
\draw [->,thick] (X4) -- (X13) node [midway,left] {$h^2$};
\draw [->,thick] (X6) -- (X15) node [midway,left] {$h^{n-1}$};
\draw [->,thick] (X7) -- (X16) node [midway,left] {$h^n$};
\draw [double,-,thick] (X8) -- (X17) node [midway,above] {};
\end{tikzpicture}
\end{center}
\end{proof}
\end{proposition}

%%%%%%%%%%%%%%%%%%%%%%%%%%%%%%%%%%%%%%%%%%
\section{$n$-cluster tilting subcategories}
Let $\mathcal{M}$ be an $n$-cluster tilting subcategory of an exact category $\mathcal{E}$. Using the result of previous section we show that every $n$-fold extension between two object in $\mathcal{M}$ is Yoneda equivalent to a unique (up to homotopy) $n$-exact sequence in $\mathcal{M}$. This is a generalization of a result by Iyama (\cite[Proposition A.1]{I1}). We also show a similar result for $n\mathbb{Z}$-cluster tilting subcategories.

\begin{proposition}$($\cite[Proposition 3.9]{EN2}$)$\label{4.1}
Let $\mathcal{M}$ be an $n$-cluster tilting subcategory of an exact category $\mathcal{E}$.
The restriction functor
$\mathbb{R}:\Mod(\mathcal{E})\rightarrow \Mod(\mathcal{M})$ induces an equivalence $\dfrac{\Mod(\mathcal{E})}{\Eff(\mathcal{E})}\overset{\widehat{\mathbb{R}}}{\simeq}\dfrac{\Mod(\mathcal{M})}{\Eff(\mathcal{M})}$ making the following diagram commutative.
\begin{center}
\begin{tikzpicture}
\node (X1) at (-4,2) {$\Mod(\mathcal{E})$};
\node (X2) at (0,2) {$\Mod(\mathcal{M})$};
\node (X3) at (-4,0) {$\dfrac{\Mod(\mathcal{E})}{\Eff(\mathcal{E})}$};
\node (X4) at (0,0) {$\dfrac{\Mod(\mathcal{M})}{\Eff(\mathcal{M})}$};
\draw [->,thick] (X1) -- (X2) node [midway,above] {$\mathbb{R}$};
\draw [->,thick] (X2) -- (X4) node [midway,right] {$\mathbb{Q}_{\mathcal{M}}$};
\draw [->,thick] (X1) -- (X3) node [midway,left] {$\mathbb{Q}_{\mathcal{E}}$};
\draw [->,thick] (X3) -- (X4) node [midway,below] {$\widehat{\mathbb{R}}$};
\end{tikzpicture}
\end{center}
\end{proposition}

\begin{theorem}\label{4.2}
Let $\mathcal{M}$ be an $n$-cluster tilting subcategory of an exact category $\mathcal{E}$ and
\begin{equation}
\xi:0\rightarrow X^0\rightarrow E^1\rightarrow E^2\rightarrow\cdots\rightarrow E^n\rightarrow X^{n+1}\rightarrow 0\notag
\end{equation}
be an acyclic sequence in $\mathcal{E}$ with $X^0,X^{n+1}\in \mathcal{M}$. Then there is a unique (up to homotopy) $n$-exact sequence
\begin{equation}
0\rightarrow X^0\rightarrow X^1\rightarrow X^2\rightarrow\cdots\rightarrow X^n\rightarrow X^{n+1}\rightarrow 0\notag
\end{equation}
that is Yoneda equivalent to $\xi$.
\begin{proof}
By the above proposition we have the following commutative diagram of functors.
\begin{center}
\begin{tikzpicture}
\node (X1) at (-4,2) {$\mathcal{M}$};
\node (X2) at (0,2) {$\mathcal{E}$};
\node (X3) at (-4,0) {$\dfrac{\Mod(\mathcal{M})}{\Eff(\mathcal{M})}$};
\node (X4) at (0,0) {$\dfrac{\Mod(\mathcal{E})}{\Eff(\mathcal{E})},$};
\draw [->,thick] (X1) -- (X2) node [midway,above] {$i$};
\draw [->,thick] (X2) -- (X4) node [midway,right] {};
\draw [->,thick] (X1) -- (X3) node [midway,left] {};
\draw [->,thick] (X3) -- (X4) node [midway,below] {$\simeq$};
\end{tikzpicture}
\end{center}
where $i$ is an inclusion functor. Then the result follows from Theorem \ref{3.5} and Proposition \ref{3.6}.
\end{proof}
\end{theorem}

In the following theorem we show that for $n\mathbb{Z}$-cluster tilting subcategories, every $kn$-fold extension is Yoneda equivalence to splicing of $k$, $n$-exact sequences.

\begin{theorem}\label{4.3}
Let $\mathcal{M}$ be an $n\mathbb{Z}$-cluster tilting subcategory of an exact category $\mathcal{E}$ and
\begin{equation}
\xi:0\rightarrow X^0\overset{f^0}{\longrightarrow} E^1\overset{f^1}{\longrightarrow} \cdots\overset{f^{kn-2}}{\longrightarrow} E^{kn-1}\overset{f^{kn-1}}{\longrightarrow} E^{kn}\overset{f^{kn}}{\longrightarrow} X^{kn+1}\rightarrow 0\notag
\end{equation}
be a $kn$-fold extension with $X^0,X^{kn+1}\in \mathcal{M}$. Then $\xi$ is Yoneda equivalent to splicing of $k$, $n$-exact sequences.
\begin{proof}
We use the induction on $k$. The case $k=1$ was proved in Theorem \ref{4.2}. Now let $k\geq 2$ and assume that the result follows for any $m$, $m<k$. Let $X^{kn}\twoheadrightarrow E^{kn}$ be a deflation with $X^{kn}\in \mathcal{M}$. So the composition $X^{kn}\twoheadrightarrow E^{kn}\twoheadrightarrow X^{kn+1}$ is a deflation. Thus we have the following commutative diagram where the top row is an $n$-exact sequence.
\begin{center}
\begin{tikzpicture}
\node (X4) at (-1.5,0.5) {$0$};
\node (X5) at (0,0.5) {$X^{(k-1)n}$};
\node (X6) at (2,0.5) {$\cdots$};
\node (X7) at (4,0.5) {$X^{kn}$};
\node (X8) at (6,0.5) {$X^{kn+1}$};
\node (X9) at (7.5,0.5) {$0$};
\node (X10) at (-7.5,-1) {$0$};
\node (X11) at (-6,-1) {$E^0$};
\node (X12) at (-4,-1) {$E^1$};
\node (X13) at (-2,-1) {$\cdots$};
\node (X14) at (0,-1) {$E^{(k-1)n}$};
\node (X15) at (2,-1) {$\cdots$};
\node (X16) at (4,-1) {$X^{kn}$};
\node (X17) at (6,-1) {$X^{kn+1}$};
\node (X18) at (7.5,-1) {$0$};
\draw [->,thick] (X4) -- (X5) node [midway,left] {};
\draw [->,thick] (X5) -- (X6) node [midway,left] {};
\draw [->,thick] (X6) -- (X7) node [midway,above] {};
\draw [->,thick] (X7) -- (X8) node [midway,left] {};
\draw [->,thick] (X8) -- (X9) node [midway,above] {};
\draw [->,thick] (X10) -- (X11) node [midway,left] {};
\draw [->,thick] (X11) -- (X12) node [midway,above] {};
\draw [->,thick] (X12) -- (X13) node [midway,above] {};
\draw [->,thick] (X13) -- (X14) node [midway,above] {};
\draw [->,thick] (X14) -- (X15) node [midway,above] {};
\draw [->,thick] (X15) -- (X16) node [midway,above] {};
\draw [->,thick] (X16) -- (X17) node [midway,above] {};
\draw [->,thick] (X17) -- (X18) node [midway,above] {};
\draw [->,thick] (X7) -- (X16) node [midway,above] {};
\draw [double,-,thick] (X8) -- (X17) node [midway,above] {};
\end{tikzpicture}
\end{center}
The proof of Theorem \ref{3.5} carries over to show that $\xi$ is Yoneda equivalent to a $kn$-fold extension of the following form.
\begin{center}
\begin{tikzpicture}
\node (X1) at (-7,0) {$0$};
\node (X2) at (-6,0) {$\bar{X^0}$};
\node (X3) at (-4.5,0) {$\bar{E^1}$};
\node (X4) at (-3,0) {$\cdots$};
\node (X5) at (-1.5,0) {$\bar{E}^{(k-1)n}$};
\node (X6) at (0,-1) {$X^{(k-1)n}$};
\node (X7) at (1.5,0) {$X^{(k-1)n+1}$};
\node (X8) at (3.5,0) {$\cdots$};
\node (X9) at (5,0) {$X^{kn}$};
\node (X10) at (6.5,0) {$X^{kn+1}$};
\node (X11) at (8,0) {$0$};
\draw [->,thick] (X1) -- (X2) node [midway,left] {};
\draw [->,thick] (X2) -- (X3) node [midway,left] {};
\draw [->,thick] (X3) -- (X4) node [midway,above] {};
\draw [->,thick] (X4) -- (X5) node [midway,left] {};
\draw [->>,thick] (X5) -- (X6) node [midway,left] {};
\draw [->,thick] (X5) -- (X7) node [midway,left] {};
\draw [>->,thick] (X6) -- (X7) node [midway,above] {};
\draw [->,thick] (X7) -- (X8) node [midway,left] {};
\draw [->,thick] (X8) -- (X9) node [midway,above] {};
\draw [->,thick] (X9) -- (X10) node [midway,above] {};
\draw [->,thick] (X10) -- (X11) node [midway,left] {};
\end{tikzpicture}
\end{center}
Then the result follows from induction hypothesis.
\end{proof}
\end{theorem}

\begin{remark}\label{4.4}
Let $\mathcal{M}$ be a small $n$-exact category. Viewing $\mathcal{M}$ as an $n$-rigid subcategory of $\mathcal{L}_1(\mathcal{M})$, by a similar argument as in the proof of Theorem \ref{4.3} we can see that the following conditions are equivalent.
\begin{itemize}
\item[(1)]
For every $i\in \{1,\cdots,kn-1\}\backslash n\mathbb{Z}$ we have $\Ext_{\mathcal{L}_1(\mathcal{M})}^i(\mathcal{M},\mathcal{M})=0$.
\item[(2)]
Every $kn$-fold extension of two object in $\mathcal{M}$ is Yoneda equivalence to splicing of $k$, $n$-exact sequences.
\end{itemize}
\end{remark}

For the proof of Theorem \ref{1.1} we need the following lemma.

\begin{lemma}\label{4.5}
Let $\mathcal{M}$ be an $n$-cluster tilting subcategory of an exact category $\mathcal{E}$ and $Y:0\rightarrow Y^0\rightarrow Y^1\rightarrow\cdots\rightarrow Y^n\rightarrow Y^{n+1}\rightarrow 0$ be an $n$-exact sequence in $\mathcal{M}$. For $j\in\{1,\cdots,n-1\}$ set $C^j:=\Imm(X^j\rightarrow X^{j+1})$ in $\mathcal{E}$. Indeed, split $Y$ into short exact sequences as follows.
%\begin{center}
\begin{diagram}\label{split}
\begin{tikzpicture}
\node (X1) at (-7.5,0) {$0$};
\node (X2) at (-6,0) {$Y^0$};
\node (X3) at (-4,0) {$Y^1$};
\node (X4) at (-2,0) {$Y^2$};
\node (X5) at (0,0) {$\cdots$};
\node (X6) at (2,0) {$Y^{n-1}$};
\node (X7) at (4,0) {$Y^n$};
\node (X8) at (6,0) {$Y^{n+1}$};
\node (X9) at (7.5,0) {$0$};
\node (X10) at (-3,-1) {$C^1$};
\node (X11) at (-1,-1) {$C^2$};
\node (X12) at (1,-1) {$C^{n-2}$};
\node (X13) at (3,-1) {$C^{n-1}$};
\draw [->,thick] (X1) -- (X2) node [midway,left] {};
\draw [->,thick] (X2) -- (X3) node [midway,left] {};
\draw [->,thick] (X3) -- (X4) node [midway,left] {};
\draw [->,thick] (X4) -- (X5) node [midway,left] {};
\draw [->,thick] (X5) -- (X6) node [midway,left] {};
\draw [->,thick] (X6) -- (X7) node [midway,above] {};
\draw [->,thick] (X7) -- (X8) node [midway,left] {};
\draw [->,thick] (X8) -- (X9) node [midway,above] {};
\draw [->>,thick] (X3) -- (X10) node [midway,left] {};
\draw [>->,thick] (X10) -- (X4) node [midway,above] {};
\draw [->>,thick] (X4) -- (X11) node [midway,above] {};
\draw [>->,thick] (X11) -- (X5) node [midway,above] {};
\draw [->>,thick] (X5) -- (X12) node [midway,above] {};
\draw [>->,thick] (X12) -- (X6) node [midway,above] {};
\draw [->>,thick] (X6) -- (X13) node [midway,above] {};
\draw [>->,thick] (X13) -- (X7) node [midway,above] {};
\end{tikzpicture}
\end{diagram}
%\end{center}
\begin{itemize}
\item[(1)]
Let $k\in \{1,\cdots,n-1\}$ such that $\Ext_{\mathcal{E}}^k(\mathcal{M},C^j)\neq 0$, then $k=n-j$.
\item[(2)]
Assume that $\Ext_{\mathcal{E}}^i(\mathcal{M},\mathcal{M})=0$ for every $i\in \{n+1,\cdots,2n-1\}$.
Let $k\in \{n+1,\cdots,2n-1\}$ such that $\Ext_{\mathcal{E}}^k(\mathcal{M},C^j)\neq 0$, then $k=2n-j$.
\item[(3)]
More generally let $\mathcal{M}$ be an $n\mathbb{Z}$-cluster tilting subcategory.
If $\Ext_{\mathcal{E}}^k(\mathcal{M},C^j)\neq 0$, then $k\in n\mathbb{Z}$ or $k\in n\mathbb{Z}\backslash \{j\}$.
\end{itemize}
\begin{proof}
For $j\in\{1,\cdots,n-2\}$ by applying the functor $\Hom_{\mathcal{E}}(X,-)$ to the short exact sequence
\begin{equation}
0\rightarrow C^j\rightarrow Y^{j+1}\rightarrow C^{j+1}\rightarrow 0, \notag
\end{equation}
we have the exact sequence
\begin{equation}
\Hom_{\mathcal{E}}(X,Y^{j+1})\rightarrow \Hom_{\mathcal{E}}(X,C^{j+1})\rightarrow\Ext_{\mathcal{E}}^1(X,C^{j})\rightarrow \Ext_{\mathcal{E}}^1(X,Y^{j+1})=0. \notag
\end{equation}
Since $Y^{j+1}\rightarrow C^{j+1}$ is a right $\mathcal{M}$-approximation, the first map is epimorphism. Thus $\Ext_{\mathcal{E}}^1(X,C^j)=0$ for $j\in\{1,\cdots,n-2\}$. The rest of proof is straightforward and we leave it to the reader.
\end{proof}
\end{lemma}

Now we are ready to prove Theorem \ref{1.1}, which gives a characterization of $n\mathbb{Z}$-cluster tilting subcategories.

\begin{proof} [Proof of Theorem \ref{1.1}]
We first show that $(1)\Rightarrow (2)$.
By \cite[Proposition 2.2]{JK} we have the following exact sequence
\begin{align*}
0&\rightarrow \Hom_{\mathcal{E}}(X,Y^0)\rightarrow \Hom_{\mathcal{E}}(X,Y^1)\rightarrow\cdots\rightarrow \Hom_{\mathcal{E}}(X,Y^n)\rightarrow\Hom_{\mathcal{E}}(X,Y^{n+1}) \\
&\rightarrow \Ext_{\mathcal{E}}^n(X,Y^0)\rightarrow \Ext_{\mathcal{E}}^n(X,Y^1).
\end{align*}
Also, using the argument in the proof of \cite[Proposition 2.2]{JK}, it is not hard to see that for every positive integer $k$ we have the following exact sequence.
\begin{align*}
\Ext_{\mathcal{E}}^{(k-1)n}(X,Y^n)\rightarrow\Ext_{\mathcal{E}}^{(k-1)n}(X,Y^{n+1})
\rightarrow \Ext_{\mathcal{E}}^{kn}(X,Y^0)\rightarrow \Ext_{\mathcal{E}}^{kn}(X,Y^1).
\end{align*}
Now by looking at Diagram 4.1 and using Lemma \ref{4.5} we have the following exact sequences
\begin{align*}
&\Ext_{\mathcal{E}}^n(X,Y^0)\rightarrow \Ext_{\mathcal{E}}^n(X,Y^1)\rightarrow\Ext_{\mathcal{E}}^n(X,C^1)\rightarrow\Ext_{\mathcal{E}}^{n+1}(X,Y^0)=0, \\
&0=\Ext_{\mathcal{E}}^{n-1}(X,C^2)\rightarrow \Ext_{\mathcal{E}}^n(X,C^1)\rightarrow\Ext_{\mathcal{E}}^{n}(X,Y^2)\rightarrow\Ext_{\mathcal{E}}^{n}(X,C^2)\rightarrow \Ext_{\mathcal{E}}^{n+1}(X,C^1)=0, \\
&\vdots
\end{align*}
Splicing these short exact sequences we obtain desire long exact sequence.

Now we show that $(2)\Rightarrow (1)$. Assume on a contrary that there exist $m\in \mathbb{Z}\backslash n\mathbb{Z}$ and $M, N\in \mathcal{M}$ such that $\Ext_{\mathcal{E}}^{m}(M, N)\neq 0$. Let $t=kn+i$ be the smallest positive integer such that $k\geq 1$, $1\leq i\leq n-1$ and there exist two objects $X,Y\in \mathcal{M}$ with $\Ext_{\mathcal{E}}^{t}(X,Y)\neq 0$.
Choose a nonzero $t$-fold extension $\xi:0\rightarrow Y\rightarrow E^1\rightarrow\cdots\rightarrow E^t\rightarrow X\rightarrow 0$ in $\Ext_{\mathcal{E}}^{t}(X,Y)$. Let $E^1\rightarrowtail Y^1$ be an inflation with $X^1\in \mathcal{M}$. Embed the inflation $Y\rightarrowtail E^1\rightarrowtail Y^1$ in an $n$-exact sequence, we obtain the following commutative diagram.
\begin{center}
\begin{tikzpicture}
\node (X1) at (-7.5,0.5) {$0$};
\node (X2) at (-6,0.5) {$Y$};
\node (X3) at (-4.5,0.5) {$E^1$};
\node (X4) at (-3,0.5) {$E^2$};
\node (X5) at (-1.5,0.5) {$\cdots$};
\node (X6) at (0,0.5) {$E^{n+1}$};
\node (X7) at (1.5,0.5) {$\cdots$};
\node (X8) at (3,0.5) {$E^t$};
\node (X9) at (4.5,0.5) {$X$};
\node (X10) at (6,0.5) {$0$};
\node (X11) at (-7.5,-1) {$0$};
\node (X12) at (-6,-1) {$Y$};
\node (X13) at (-4.5,-1) {$Y^1$};
\node (X14) at (-3,-1) {$Y^2$};
\node (X15) at (-1.5,-1) {$\cdots$};
\node (X16) at (0,-1) {$Y^{n+1}$};
\node (X17) at (1.5,-1) {$0$};
%\node (X18) at (3,-1) {$0$};
\draw [->,thick] (X1) -- (X2) node [midway,left] {};
\draw [->,thick] (X2) -- (X3) node [midway,left] {};
\draw [->,thick] (X3) -- (X4) node [midway,left] {};
\draw [->,thick] (X4) -- (X5) node [midway,left] {};
\draw [->,thick] (X5) -- (X6) node [midway,left] {};
\draw [->,thick] (X6) -- (X7) node [midway,above] {};
\draw [->,thick] (X7) -- (X8) node [midway,left] {};
\draw [->,thick] (X8) -- (X9) node [midway,above] {};
\draw [->,thick] (X9) -- (X10) node [midway,above] {};
\draw [->,thick] (X11) -- (X12) node [midway,above] {};
\draw [->,thick] (X12) -- (X13) node [midway,above] {};
\draw [->,thick] (X13) -- (X14) node [midway,above] {};
\draw [->,thick] (X14) -- (X15) node [midway,above] {};
\draw [->,thick] (X15) -- (X16) node [midway,above] {};
\draw [->,thick] (X16) -- (X17) node [midway,above] {};
%\draw [->,thick] (X17) -- (X18) node [midway,above] {};
\draw [double,-,thick] (X2) -- (X12) node [midway,left] {};
\draw [->,thick] (X3) -- (X13) node [midway,left] {};
\end{tikzpicture}
\end{center}
Using the dual of Remark \ref{3.4}, we can obtain the following diagram where the top row is Yoneda equivalent to $\xi$.
\begin{center}
\begin{tikzpicture}
\node (X1) at (-7.5,0.5) {$0$};
\node (X2) at (-6,0.5) {$Y$};
\node (X3) at (-4.5,0.5) {$Y^1$};
\node (X4) at (-3,0.5) {$E^2$};
\node (X5) at (-1.5,0.5) {$\cdots$};
\node (X6) at (0,0.5) {$E^{n+1}$};
\node (X7) at (1.5,0.5) {$\cdots$};
\node (X8) at (3,0.5) {$E^t$};
\node (X9) at (4.5,0.5) {$X$};
\node (X10) at (6,0.5) {$0$};
\node (X11) at (-7.5,-1) {$0$};
\node (X12) at (-6,-1) {$Y$};
\node (X13) at (-4.5,-1) {$Y^1$};
\node (X14) at (-3,-1) {$Y^2$};
\node (X15) at (-1.5,-1) {$\cdots$};
\node (X16) at (0,-1) {$Y^{n+1}$};
\node (X17) at (1.5,-1) {$0$};
%\node (X18) at (3,-1) {$0$};
\draw [->,thick] (X1) -- (X2) node [midway,left] {};
\draw [->,thick] (X2) -- (X3) node [midway,left] {};
\draw [->,thick] (X3) -- (X4) node [midway,left] {};
\draw [->,thick] (X4) -- (X5) node [midway,left] {};
\draw [->,thick] (X5) -- (X6) node [midway,left] {};
\draw [->,thick] (X6) -- (X7) node [midway,above] {};
\draw [->,thick] (X7) -- (X8) node [midway,left] {};
\draw [->,thick] (X8) -- (X9) node [midway,above] {};
\draw [->,thick] (X9) -- (X10) node [midway,above] {};
\draw [->,thick] (X11) -- (X12) node [midway,above] {};
\draw [->,thick] (X12) -- (X13) node [midway,above] {};
\draw [->,thick] (X13) -- (X14) node [midway,above] {};
\draw [->,thick] (X14) -- (X15) node [midway,above] {};
\draw [->,thick] (X15) -- (X16) node [midway,above] {};
\draw [->,thick] (X16) -- (X17) node [midway,above] {};
%\draw [->,thick] (X17) -- (X18) node [midway,above] {};
\draw [double,-,thick] (X2) -- (X12) node [midway,left] {};
\draw [double,-,thick] (X3) -- (X13) node [midway,left] {};
\end{tikzpicture}
\end{center}
Note that in this diagram $E^2$ is differ from $E^2$ in the previous diagram, but for simplicity we use the same notation. By the minimality of $t$ it is easy to see that the sequence
\begin{align*}
0&\rightarrow \Hom_{\mathcal{E}}(X,Y^0)\rightarrow \Hom_{\mathcal{E}}(X,Y^1)\rightarrow\cdots\rightarrow \Hom_{\mathcal{E}}(X,Y^n)\rightarrow\Hom_{\mathcal{E}}(X,Y^{n+1}) \\
&\rightarrow \cdots \\
&\rightarrow \Ext_{\mathcal{E}}^{kn}(X,Y^0)\rightarrow \cdots\rightarrow \Ext_{\mathcal{E}}^{kn}(X,Y^{i-1})\rightarrow\Ext_{\mathcal{E}}^{kn}(X,Y^{i}),
\end{align*}
is exact. Now consider the Diagram 4.1 for $Y$. By using the minimality of $t$ we have
\begin{align*}
&\Ext_{\mathcal{E}}^r(X,C^1)=0, r\in \{kn+1,\cdots,kn+i-2\}\\
&\Ext_{\mathcal{E}}^r(X,C^2)=0, r\in \{kn+1,\cdots,kn+i-3\}\\
&\vdots \\
&\Ext_{\mathcal{E}}^r(X,C^{i-2})=0, r\in \{kn+1\}.
\end{align*}
Using this equations and by the dual argument of the proof of Theorem \ref{3.5} we obtain the following commutative diagram with exact rows.
\begin{center}
\begin{tikzpicture}
\node (X1) at (-7,0.5) {$0$};
\node (X2) at (-6,0.5) {$Y$};
\node (X3) at (-4.5,0.5) {$\cdots$};
\node (X4) at (-3,0.5) {$Y^{i-1}$};
\node (X5) at (-1.5,0.5) {$Y^i$};
\node (X6) at (0,0.5) {$E^{i+1}$};
\node (X7) at (1.5,0.5) {$\cdots$};
\node (X8) at (3,0.5) {$E^{n+1}$};
\node (X9) at (4.5,0.5) {$\cdots$};
\node (X10) at (6,0.5) {$E^t$};
\node (X11) at (7.5,0.5) {$X$};
\node (X12) at (8.5,0.5) {$0$};
\node (X13) at (-7,-1) {$0$};
\node (X14) at (-6,-1) {$Y$};
\node (X15) at (-4.5,-1) {$\cdots$};
\node (X16) at (-3,-1) {$Y^{i-1}$};
\node (X17) at (-1.5,-1) {$Y^i$};
\node (X18) at (0,-1) {$Y^{i+1}$};
\node (X19) at (1.5,-1) {$\cdots$};
\node (X20) at (3,-1) {$Y^{n+1}$};
\node (X21) at (4.5,-1) {$0$};
\node (X22) at (-3.75,-2) {$C^{i-2}$};
\node (X23) at (-2.25,-2) {$C^{i-1}$};
\node (X24) at (-0.75,-2) {$C^i$};
\node (X25) at (0.75,-2) {$C^{i+1}$};
\draw [->,thick] (X1) -- (X2) node [midway,left] {};
\draw [->,thick] (X2) -- (X3) node [midway,left] {};
\draw [->,thick] (X3) -- (X4) node [midway,left] {};
\draw [->,thick] (X4) -- (X5) node [midway,left] {};
\draw [->,thick] (X5) -- (X6) node [midway,left] {};
\draw [->,thick] (X6) -- (X7) node [midway,above] {};
\draw [->,thick] (X7) -- (X8) node [midway,left] {};
\draw [->,thick] (X8) -- (X9) node [midway,above] {};
\draw [->,thick] (X9) -- (X10) node [midway,above] {};
\draw [->,thick] (X10) -- (X11) node [midway,above] {};
\draw [->,thick] (X11) -- (X12) node [midway,above] {};
\draw [->,thick] (X13) -- (X14) node [midway,above] {};
\draw [->,thick] (X14) -- (X15) node [midway,above] {};
\draw [->,thick] (X15) -- (X16) node [midway,above] {};
\draw [->,thick] (X16) -- (X17) node [midway,above] {};
\draw [->,thick] (X17) -- (X18) node [midway,above] {};
\draw [->,thick] (X18) -- (X19) node [midway,above] {};
\draw [->,thick] (X19) -- (X20) node [midway,above] {};
\draw [->,thick] (X20) -- (X21) node [midway,above] {};
\draw [double,-,thick] (X2) -- (X14) node [midway,left] {};
\draw [double,-,thick] (X3) -- (X15) node [midway,left] {};
\draw [double,-,thick] (X4) -- (X16) node [midway,left] {};
\draw [double,-,thick] (X5) -- (X17) node [midway,left] {};
\draw [->>,thick] (X15) -- (X22) node [midway,above] {};
\draw [>->,thick] (X22) -- (X16) node [midway,above] {};
\draw [->>,thick] (X16) -- (X23) node [midway,above] {};
\draw [>->,thick] (X23) -- (X17) node [midway,above] {};
\draw [->>,thick] (X17) -- (X24) node [midway,above] {};
\draw [>->,thick] (X24) -- (X18) node [midway,above] {};
\draw [->>,thick] (X18) -- (X25) node [midway,above] {};
\draw [>->,thick] (X25) -- (X19) node [midway,above] {};
\end{tikzpicture}
\end{center}
For simplicity we set $f^i:Y^i\rightarrow Y^{i+1}$ as composition $g^i:Y^i\twoheadrightarrow C^i$ with $h^i:C^i\rightarrowtail Y^{i+1}$. We claim that the induced sequence
\begin{align*}
\Ext_{\mathcal{E}}^{kn}(X,Y^{i-1})\rightarrow  \Ext_{\mathcal{E}}^{kn}(X,Y^{i})\rightarrow  \Ext_{\mathcal{E}}^{kn}(X,Y^{i+1}),
\end{align*}
is not exact. First by applying $\Hom_{\mathcal{E}}(X,-)$ to the short exact sequence $C^{i-1}\rightarrowtail Y^{i}\twoheadrightarrow C^{i}$ we have the following exact sequence of abelian groups.
\begin{align*}
\Ext_{\mathcal{E}}^{kn}(X,Y^{i})\rightarrow  \Ext_{\mathcal{E}}^{kn}(X,C^{i})\overset{\theta}{\rightarrow}\Ext_{\mathcal{E}}^{kn+1}(X,C^{i-1})\rightarrow \Ext_{\mathcal{E}}^{kn+1}(X,Y^{i-1}).
\end{align*}
Now consider the nonzero element
\begin{center}
$\xi:0\rightarrow C^i\rightarrow E^{i+1}\rightarrow \cdots \rightarrow C^t\rightarrow X\rightarrow 0,$
\end{center}
in $\Ext_{\mathcal{E}}^{kn}(X,C^{i})$. By Remark \ref{3.3},
\begin{center}
$\eta :=\theta (\xi)=[0\rightarrow C^{i-1}\rightarrow Y^i\rightarrow C^i\rightarrow 0] \xi .$
\end{center}
Obviously $\Ext_{\mathcal{E}}^{kn+1}(X,h^{i-1})(\eta)=0$. Thus $\Ext_{\mathcal{E}}^{kn}(X,g^i)$ is not an epimorphism. But by applying $\Hom_{\mathcal{E}}(X,-)$ to the short exact sequence $C^{i}\rightarrowtail Y^{i+1}\twoheadrightarrow C^{i+1}$ we can see that
\begin{center}
$\Ext_{\mathcal{E}}^{kn}(X,C^i)=\Ker \big(\Ext_{\mathcal{E}}^{kn}(X,g^{i+1}) \big) \subseteq \Ker \big(\Ext_{\mathcal{E}}^{kn}(X,f^{i+1}) \big).$
\end{center}
So we have
\begin{center}
$\Imm \big(\Ext^{kn}(X,f^{i}) \big) \neq \Ker \big(\Ext^{kn}(X,f^{i+1}) \big),$
\end{center}
which gives a contradiction and the result follows.
\end{proof}

\begin{remark}\label{4.6}
Let $\mathcal{M}$ be an $n$-rigid subcategory of an exact category $\mathcal{E}$. By the proof of Theorem \ref{1.1} the following conditions are equivalent.
\begin{itemize}
\item[(1)]
For every $j\in \{1,\cdots,kn+i-1\}\backslash n\mathbb{Z}$ where $1\leq i\leq n-1$ we have $\Ext_{\mathcal{E}}^j(\mathcal{M},\mathcal{M})=0$.
\item[(2)]
For every $X\in \mathcal{M}$ and every $n$-exact sequence $Y:Y^0\rightarrow Y^1\rightarrow \cdots Y^n\rightarrow Y^{n+1}$ the following induced sequence of abelian groups is exact.
\begin{align*}
0&\rightarrow \Hom_{\mathcal{E}}(X,Y^0)\rightarrow \Hom_{\mathcal{E}}(X,Y^1)\rightarrow\cdots\rightarrow \Hom_{\mathcal{E}}(X,Y^n)\rightarrow\Hom_{\mathcal{E}}(X,Y^{n+1}) \\
&\rightarrow \cdots \\
&\rightarrow \Ext_{\mathcal{E}}^{kn}(X,Y^0)\rightarrow \cdots\rightarrow \Ext_{\mathcal{E}}^{kn}(X,Y^{i-1})\rightarrow\Ext_{\mathcal{E}}^{kn}(X,Y^{i}).
\end{align*}
\end{itemize}
\end{remark}

\section{$n\mathbb{Z}$-abelian and $n\mathbb{Z}$-exact categories}
In this section, inspired by the results of the previous section, we define $n\mathbb{Z}$-abelian and $n\mathbb{Z}$-exact categories and show that they are axiomatisation of $n\mathbb{Z}$-cluster tilting subcategories of abelian and exact categories, respectively.

Let $\mathcal{M}$ be an $n$-exact category. For an object $M\in \mathcal{M}$ we denote the diagonal and codiagonal map by
\begin{align*}
&\Delta_M=\begin{pmatrix}
1\\
1
\end{pmatrix}:M\longrightarrow M\oplus M\\
and\\
&\nabla_M=\begin{pmatrix}
1 & 1
\end{pmatrix}:M\oplus M\longrightarrow  M,
\end{align*}
respectively. We writ $\Delta$ and $\nabla$ when $M$ is clear from the context. Now let $L,N\in \mathcal{M}$. An $n$-exact sequence
\begin{center}
$\xi:0\rightarrow L\rightarrow M^1\rightarrow M^2\rightarrow \cdots\rightarrow M^n\rightarrow N\rightarrow 0$
\end{center}
is called an $n$-extension of $N$ by $L$. Let $f:L\rightarrow L'$ be an arbitrary morphism. By taking $n$-pushout along $f$
we obtain the following morphism between $n$-exact sequences.
\begin{center}
\begin{tikzpicture}
%\node (X1) at (-8,1) {$0$};
%\node (X2) at (-6.5,1) {$X^0$};
\node (X3) at (-5,1) {$0$};
\node (X4) at (-3.5,1) {$L$};
\node (X5) at (-1.75,1) {$M^1$};
\node (X6) at (0,1) {$\cdots$};
\node (X7) at (1.5,1) {$M^n$};
\node (X8) at (3,1) {$N$};
\node (X9) at (4.5,1) {$0$};
\node (X10) at (-5,-0.5) {$0$};
\node (X11) at (-3.5,-0.5) {$L'$};
\node (X12) at (-1.75,-0.5) {$W^1$};
\node (X13) at (0,-0.5) {$\cdots$};
\node (X14) at (1.5,-0.5) {$W^n$};
\node (X15) at (3,-0.5) {$N$};
\node (X16) at (4.5,-0.5) {$0.$};
%\draw [->,thick] (X1) -- (X2) node [midway,left] {};
%\draw [->,thick] (X2) -- (X3) node [midway,left] {};
\draw [->,thick] (X3) -- (X4) node [midway,left] {};
\draw [->,thick] (X4) -- (X5) node [midway,left] {};
\draw [->,thick] (X5) -- (X6) node [midway,left] {};
\draw [->,thick] (X6) -- (X7) node [midway,above] {};
\draw [->,thick] (X7) -- (X8) node [midway,left] {};
\draw [->,thick] (X8) -- (X9) node [midway,above] {};
\draw [->,thick] (X10) -- (X11) node [midway,left] {};
\draw [->,thick] (X11) -- (X12) node [midway,above] {};
\draw [->,thick] (X12) -- (X13) node [midway,above] {};
\draw [->,thick] (X13) -- (X14) node [midway,above] {};
\draw [->,thick] (X14) -- (X15) node [midway,above] {};
\draw [->,thick] (X15) -- (X16) node [midway,above] {};
\draw [->,thick] (X4) -- (X11) node [midway,left] {$f$};
\draw [->,thick] (X5) -- (X12) node [midway,above] {};
\draw [->,thick] (X7) -- (X14) node [midway,above] {};
\draw [double,-,thick] (X8) -- (X15) node [midway,above] {};
\end{tikzpicture}
\end{center}
The bottom row is denoted by $f\xi$. For a morphism $g:N'\rightarrow N$, $\xi g$ is defined dually by taking $n$-pullback along $g$.

For another $n$-extension of $N$ by $L$ like
\begin{equation}
\xi':0\rightarrow L\rightarrow M'^1\rightarrow M'^2\rightarrow \cdots\rightarrow M'^n\rightarrow N\rightarrow 0\notag
\end{equation}
their Baer sum is defined as $\nabla (\xi\oplus\xi')\Delta$. Two $n$-extensions of $N$ by $L$ are said to be Yoneda equivalent if there is a morphism of $n$-exact sequences with identity end terms, from one to another. By \cite[Proposition 4.10 ]{J} this is an equivalence relation and we denote by $\nExt_{\mathcal{M}}^1(M,L)$ the set of Yoneda equivalence classes of $n$-extension of $N$ by $L$. This is an abelian group (ignoring set theoretical difficulties) by the Baer sum operation. Also for a positive integer $k$, $k$-fold $n$-extension is defined similarly. Indeed a $k$-fold $n$-extension is splicing of $k$, $n$-extensions. The Yoneda equivalence classes of $k$-fold $n$-extension of $N$ by $L$ is denoted by $\nExt_{\mathcal{M}}^k(M,L)$ and it is an abelian group with Baer sum operation (for more details see \cite[Section 5.2]{Lu}).

Now let $\mathcal{M}$ be an $n$-exact category and
\begin{equation}\label{nexact}
0\rightarrow L\rightarrow M^1\rightarrow M^2\rightarrow \cdots\rightarrow M^n\rightarrow N\rightarrow 0
\end{equation}
be an $n$-exact sequence in $\mathcal{M}$. For every object $X\in \mathcal{M}$ there is the following induced sequence of abelian groups.
\begin{align}\label{long}
0&\rightarrow \Hom_{\mathcal{M}}(X,L)\rightarrow \Hom_{\mathcal{M}}(X,M^1)\rightarrow\cdots\rightarrow \Hom_{\mathcal{M}}(X,M^n)\rightarrow\Hom_{\mathcal{M}}(X,N) \notag\\
&\rightarrow \nExt_{\mathcal{M}}^1(X,L)\rightarrow \nExt_{\mathcal{M}}^1(X,M^1)\rightarrow\cdots\rightarrow \nExt_{\mathcal{M}}^1(X,M^n)\rightarrow \nExt_{\mathcal{M}}^1(X,N)\notag \\
&\rightarrow \nExt_{\mathcal{M}}^2(X,L)\rightarrow \nExt_{\mathcal{M}}^{2}(X,M^1)\rightarrow\cdots\rightarrow \nExt_{\mathcal{M}}^{2}(X,M^n)\rightarrow \nExt_{\mathcal{M}}^{2}(X,N)\notag \\
&\rightarrow \cdots
\end{align}

Dually, we have the following induced sequence of abelian groups.
\begin{align}\label{long1}
0&\rightarrow \Hom_{\mathcal{M}}(N,X)\rightarrow \Hom_{\mathcal{M}}(M^n,X)\rightarrow\cdots\rightarrow \Hom_{\mathcal{M}}(M^1,X)\rightarrow\Hom_{\mathcal{M}}(L,X) \notag\\
&\rightarrow \nExt_{\mathcal{M}}^1(N,X)\rightarrow \nExt_{\mathcal{M}}^1(M^n,X)\rightarrow\cdots\rightarrow \nExt_{\mathcal{M}}^1(M^1,X)\rightarrow \nExt_{\mathcal{M}}^1(L,X)\notag \\
&\rightarrow \nExt_{\mathcal{M}}^2(N,X)\rightarrow \nExt_{\mathcal{M}}^{2}(M^n,X)\rightarrow\cdots\rightarrow \nExt_{\mathcal{M}}^{2}(M^1,X)\rightarrow \nExt_{\mathcal{M}}^{2}(L,X)\notag \\
&\rightarrow \cdots
\end{align}

It is natural to ask about the exactness of these sequences. The following theorem gives the answer.

\begin{theorem}\label{5.1}
Let $\mathcal{M}$ be an $n$-exact category realized as an $n$-cluster tilting subcategory of an exact category $\mathcal{E}$. Then the following conditions are equivalent.
\begin{itemize}
\item[(1)]
$\mathcal{M}$ is an $n\mathbb{Z}$-cluster tilting subcategory of $\mathcal{E}$.
\item[(2)]
For every $n$-exact sequence like \eqref{nexact} and every $X\in \mathcal{M}$ the induced sequence of abelian groups \eqref{long} is exact.
\item[(3)]
For every $n$-exact sequence like \eqref{nexact} and every $X\in \mathcal{M}$ the induced sequence of abelian groups \eqref{long1} is exact.
\end{itemize}
\begin{proof}
We only prove that $(1)$ and $(2)$ are equivalent. The equivalence of $(1)$ and $(3)$ is dual.

Because $\mathcal{M}$ is an $n\mathbb{Z}$-cluster tilting subcategory, by Theorem \ref{4.3} we have $\Ext_{\mathcal{E}}^{kn}(X,Y)\cong \nExt_{\mathcal{M}}^k(X,Y)$. Thus $(1)\Rightarrow (2)$ follows from Theorem \ref{1.1}.

$(2)\Rightarrow (1)$. By Theorem \ref{4.2}, $\Ext_{\mathcal{E}}^{n}(X,Y)\cong \nExt_{\mathcal{M}}^1(X,Y)$ and hence we have the following exact sequence of abelian groups.
\begin{align*}
0&\rightarrow \Hom_{\mathcal{E}}(X,L)\rightarrow \Hom_{\mathcal{E}}(X,M^1)\rightarrow\cdots\rightarrow \Hom_{\mathcal{E}}(X,M^n)\rightarrow\Hom_{\mathcal{E}}(X,M^{n+1}) \\
&\rightarrow \Ext_{\mathcal{E}}^n(X,L)\rightarrow \Ext_{\mathcal{E}}^n(X,M^1)\rightarrow\cdots\rightarrow \Ext_{\mathcal{E}}^n(X,M^n)\rightarrow \Ext_{\mathcal{E}}^n(X,M^{n+1}).
\end{align*}
Therefor by Remark \ref{4.6}, for every $k\in \{n+1,\cdots,2n-1\}$ we have $\Ext_{\mathcal{E}}^k(\mathcal{M},\mathcal{M})=0$. But this implies that $\Ext_{\mathcal{E}}^{2n}(X,Y)\cong \nExt_{\mathcal{M}}^2(X,Y)$ by Remark \ref{4.4}. So we have the following exact sequence
\begin{align*}
0&\rightarrow \Hom_{\mathcal{E}}(X,L)\rightarrow \Hom_{\mathcal{E}}(X,M^1)\rightarrow\cdots\rightarrow \Hom_{\mathcal{E}}(X,M^n)\rightarrow\Hom_{\mathcal{E}}(X,M^{n+1}) \\
&\rightarrow \Ext_{\mathcal{E}}^n(X,L)\rightarrow \Ext_{\mathcal{E}}^n(X,M^1)\rightarrow\cdots\rightarrow \Ext_{\mathcal{E}}^n(X,M^n)\rightarrow \Ext_{\mathcal{E}}^n(X,M^{n+1})\\
&\rightarrow \Ext_{\mathcal{E}}^{2n}(X,L)\rightarrow \Ext_{\mathcal{E}}^{2n}(X,M^1)\rightarrow\cdots\rightarrow \Ext_{\mathcal{E}}^{2n}(X,M^n)\rightarrow \Ext_{\mathcal{E}}^{2n}(X,M^{n+1}).
\end{align*}
The result follows by repeating this argument.
\end{proof}
\end{theorem}

Motivated by Theorem \ref{5.1} we give the following definition as axiomatisation of $n\mathbb{Z}$-cluster tilting subcategories.

\begin{definition}\label{5.2}
Let $\mathcal{M}$ be an $n$-exact (resp. $n$-abelian) category. We say that $\mathcal{M}$ is an $n\mathbb{Z}$-exact (resp. $n\mathbb{Z}$-abelian) category if it satisfies the following two conditions.
\begin{itemize}
\item[(1)]
For every $n$-exact sequence like \eqref{nexact} and every $X\in \mathcal{M}$ the induced sequence of abelian groups \eqref{long} is exact.
\item[(2)]
For every $n$-exact sequence like \eqref{nexact} and every $X\in \mathcal{M}$ the induced sequence of abelian groups \eqref{long1} is exact.
\end{itemize}
\end{definition}

Let $\mathcal{M}$ be a small $n$-abelian category. A functor $F\in \Mod(\mathcal{M})$ is called {\em finitely presented} or {\em coherent}, if there exists an exact sequence of the form
\begin{center}
$\Hom_\mathcal{M}(-,X)\overset{\Hom_\mathcal{M}(-,f)}{\longrightarrow} \Hom_\mathcal{M}(-,Y)\rightarrow F\rightarrow 0.$
\end{center}
We denote by $\modd(\mathcal{M})$ the full subcategory of $\Mod(\mathcal{M})$ consist of all finitely presented functors. Since every morphism in $\mathcal{M}$ has a weak kernel, $\rm{mod}(\mathcal{M})$ is an abelian category \cite[Theorem 1.4]{Fr2}.

A functor $F\in \modd(\mathcal{M})$ is called {\em effaceable}, if there is an exact sequence
\begin{center}
$\Hom_\mathcal{M}(-,X)\overset{\Hom_\mathcal{M}(-,f)}{\longrightarrow} \Hom_\mathcal{M}(-,Y)\rightarrow F\rightarrow 0,$
\end{center}
for some epimorphism $f:X\rightarrow Y$. The full subcategory of effaceable functors is denoted by $\eff(\mathcal{M})$.

\begin{corollary}
Let $\mathcal{M}$ be a small $n$-abelian category. Then the following statements are equivalent.
\begin{itemize}
\item[(1)]
$\mathcal{M}$ is equivalent to an $n\mathbb{Z}$-cluster tilting subcategory of $\dfrac{\modd(\mathcal{M})}{\eff(\mathcal{M})}$.
\item[(2)]
$\mathcal{M}$ is an $n\mathbb{Z}$-abelian category.
\end{itemize}
\begin{proof}
By \cite{EN2,Kv} $\mathcal{M}$ is equivalent to an $n$-cluster tilting subcategory of $\dfrac{\modd(\mathcal{M})}{\eff(\mathcal{M})}$. Thus the result follows from Theorem \ref{5.1}.
\end{proof}
\end{corollary}

By Theorem \ref{1.1}, for $n$-cluster tilting subcategories the conditions $(1)$ and $(2)$ of Definition \ref{5.2} are equivalent. So it is natural to ask about the equivalence of these conditions for all $n$-exact categories. By the following proposition for any small $n$-exact categories these conditions are equivalent.

\begin{proposition}
Let $\mathcal{M}$ be a small $n$-exact category. If we take $\mathcal{M}$ as a subcategory of $\mathcal{L}_1(\mathcal{M})$ (by Gabriel-Quillen embedding), then the following statements are equivalent.
\begin{itemize}
\item[(1)]
$\mathcal{M}$ is an $n\mathbb{Z}$-exact category.
\item[(2)]
 For every $X\in \mathcal{M}$ and every exact sequence $Y^0\rightarrow Y^1\rightarrow \cdots Y^n\rightarrow Y^{n+1}$ in $\mathcal{L}_1(\mathcal{M})$ with terms in $\mathcal{M}$ the the following induced sequence of abelian groups is exact.
\begin{align*}
0&\rightarrow \Hom_{\mathcal{L}_1(\mathcal{M})}(X,Y^0)\rightarrow \Hom_{\mathcal{L}_1(\mathcal{M})}(X,Y^1)\rightarrow\cdots\rightarrow \Hom_{\mathcal{L}_1(\mathcal{M})}(X,Y^n)\rightarrow\Hom_{\mathcal{L}_1(\mathcal{M})}(X,Y^{n+1}) \\
&\rightarrow \Ext_{\mathcal{L}_1(\mathcal{M})}^n(X,Y^0)\rightarrow \Ext_{\mathcal{L}_1(\mathcal{M})}^n(X,Y^1)\rightarrow\cdots\rightarrow \Ext_{\mathcal{L}_1(\mathcal{M})}^n(X,Y^n)\rightarrow\Ext_{\mathcal{L}_1(\mathcal{M})}^n(X,Y^{n+1}) \\
&\rightarrow \Ext_{\mathcal{L}_1(\mathcal{M})}^{2n}(X,Y^0)\rightarrow \Ext_{\mathcal{L}_1(\mathcal{M})}^{2n}(X,Y^1)\rightarrow\cdots\rightarrow \Ext_{\mathcal{L}_1(\mathcal{M})}^{2n}(X,Y^n)\rightarrow\Ext_{\mathcal{L}_1(\mathcal{M})}^{2n}(X,Y^{n+1}) \\
&\rightarrow \cdots
\end{align*}
\item[(3)]
 For every $X\in \mathcal{M}$ and every exact sequence $Y^0\rightarrow Y^1\rightarrow \cdots Y^n\rightarrow Y^{n+1}$ in $\mathcal{L}_1(\mathcal{M})$ with terms in $\mathcal{M}$ the the following induced sequence of abelian groups is exact.
\begin{align*}
0&\rightarrow \Hom_{\mathcal{L}_1(\mathcal{M})}(Y^{n+1},X)\rightarrow \Hom_{\mathcal{L}_1(\mathcal{M})}(Y^n,X)\rightarrow\cdots\rightarrow \Hom_{\mathcal{L}_1(\mathcal{M})}(Y^1,X)\rightarrow\Hom_{\mathcal{L}_1(\mathcal{M})}(Y^{0},X) \\
&\rightarrow \Ext_{\mathcal{L}_1(\mathcal{M})}^n(Y^{n+1},X)\rightarrow \Ext_{\mathcal{L}_1(\mathcal{M})}^n(Y^n,X)\rightarrow\cdots\rightarrow \Ext_{\mathcal{L}_1(\mathcal{M})}^n(Y^1,X)\rightarrow\Ext_{\mathcal{L}_1(\mathcal{M})}^n(Y^{0},X) \\
&\rightarrow \Ext_{\mathcal{L}_1(\mathcal{M})}^{2n}(Y^{n+1},X)\rightarrow \Ext_{\mathcal{L}_1(\mathcal{M})}^{2n}(Y^n,X)\rightarrow\cdots\rightarrow \Ext_{\mathcal{L}_1(\mathcal{M})}^{2n}(Y^1,X)\rightarrow\Ext_{\mathcal{L}_1(\mathcal{M})}^{2n}(Y^{0},X) \\
&\rightarrow \cdots
\end{align*}
\end{itemize}
\begin{proof}
The proof is similar to the proof of Theorem \ref{5.1} and is left to the reader.
\end{proof}
\end{proposition}

\section*{acknowledgements}
The research of first author was in part supported by a grant from IPM. The research of the second author was in part supported by a grant from IPM (No. 1400170417).

\end{document}